\documentclass[12pt]{amsart}
\usepackage{amsfonts}
\usepackage{amsmath}
\usepackage{amssymb}
\usepackage[margin=1in]{geometry}
\usepackage{amsthm}   							
\usepackage{mathtools}							

\usepackage[english]{babel}
\usepackage[hidelinks]{hyperref}

\makeatletter
\@namedef{subjclassname@2020}{\textup{2020} Mathematics Subject Classification}
\makeatother

\newtheorem{theorem}{Theorem}[section]
\newtheorem{corollary}[theorem]{Corollary}
\newtheorem{lemma}[theorem]{Lemma}
\newtheorem{proposition}[theorem]{Proposition}

\theoremstyle{definition}

\newtheorem{remark}[theorem]{Remark}

\numberwithin{equation}{section}

\newcommand{\dif}{\mathop{}\!\mathrm{d}}

\newcommand{\N}{\mathbb{N}}
\newcommand{\Z}{\mathbb{Z}}

\newcommand{\R}{\mathbb{R}}
\newcommand{\C}{\mathbb{C}}

\renewcommand{\Re}{\operatorname{Re}}
\renewcommand{\Im}{\operatorname{Im}}

\newcommand{\I}{\mathrm{i}}
\newcommand{\e}{\mathrm{e}}
\newcommand{\mc}{\mathrm{c}}

\newcommand{\eps}{\varepsilon}
\newcommand{\vphi}{\varphi}

\newcommand{\MP}{\mathcal{P}}
\newcommand{\MN}{\mathcal{N}}
\newcommand{\MB}{\mathcal{B}}

\DeclareMathOperator{\Li}{Li}
\DeclareMathOperator{\supp}{supp}
\DeclareMathOperator{\res}{res}

\DeclarePairedDelimiter\abs{\lvert}{\rvert}

\begin{document}

\title[The connection between zero-free regions and the error term in the PNT]{On the connection between zero-free regions and the error term in the prime number theorem}

\author[F. Broucke]{Frederik Broucke}
\thanks{F. Broucke was supported by a postdoctoral fellowship (grant number 12ZZH23N) of the Research Foundation -- Flanders}

\address{Department of Mathematics: Analysis, Logic and Discrete Mathematics\\ Ghent University\\ Krijgslaan 281\\ 9000 Gent\\ Belgium}

\email{fabrouck.broucke@ugent.be}

\subjclass[2020]{Primary 11M26, 11N05; Secondary 11M41, 11N80}


\keywords{Zero-free region, Prime Number Theorem with remainder, Pintz's theorems on Ingham's question, Beurling generalized number systems}

\maketitle
\begin{abstract}
We provide for a wide class of zero-free regions an upper bound for the error term in the Prime Number Theorem, refining works of Pintz (1980), Johnston (2024), and R\'ev\'esz (2024). Our method does not only apply to the Riemann zeta function, but to general Beurling zeta functions. Next we construct Beurling zeta functions having infinitely many zeros on a prescribed contour, and none to the right, for a wide class of such contours. We also deduce an oscillation result for the corresponding error term in the Prime Number Theorem, showing that our aforementioned refinement is close to being sharp.
\end{abstract}

\section{Introduction}
In his 1859 paper, Riemann communicated his revolutionary insight that the distribution of prime numbers is intimately connected with the distribution of zeros of the Riemann zeta function $\zeta(s)$. His ideas led to the first proofs of the Prime Number Theorem (PNT) by Hadamard and de la Vall\'ee Poussin, which crucially make use of the fact that $\zeta(s)$ does not vanish for $\Re s \ge 1$. Later, de la Vall\'ee Poussin proved that $\zeta(s)$ does not vanish in the region $\sigma > 1 - c/\log(\abs{t}+2)$ for some $c>0$ (where we write $s=\sigma+\I t$). From this he deduced the following form of the PNT with remainder:
\[
	\psi(x) \coloneqq \sum_{n\le x}\Lambda(n) = x+ O\Bigl( x\exp\bigl(-c'\sqrt{\log x}\bigr)\Bigr), \quad \text{for some } c'>0.
\]

In his 1932 monograph \cite{In32}, Ingham posed the question of what error term would follow from a general zero-free region $\sigma > 1- \eta(\abs{t})$ for a certain function $\eta$. His initial result was greatly improved by J\'anos Pintz in 1980 \cite{Pi80}. Before stating Pintz's theorem we define, for a given function $\eta: [0,\infty) \to [0,1/2]$, the function
\[
	\omega_{\eta}(x) \coloneqq \inf_{t\ge1} \bigl(\eta(t)\log x + \log t\bigr).
\]
When $\eta$ is clear from the context, we will often omit the subscript and simply write $\omega(x)$.
\begin{theorem}[Pintz]
Suppose that $\eta: [0, \infty) \to (0,1/2]$ is continuously decreasing, and suppose that $\zeta(s)$ has no zeros in the domain $\sigma > 1- \eta(\abs{t})$. Then for every $\eps>0$,
\[
	\psi(x)-x   \ll_{\eps} x\exp\bigl(-(1-\eps)\omega(x)\bigr).
\]
\end{theorem}

In the same paper, Pintz also showed the following converse.
\begin{theorem}[Pintz]
Suppose that $\zeta(s)$ has infinitely many zeros in the domain $\sigma \ge 1-\eta(\abs{t})$, where $f(u) \coloneqq \eta(\e^{u})$ is such that $f'(u)$ is non-decreasing with limit $0$ and strictly increasing to $0$ if $f\to 0$. Then for every $\eps>0$,
\[
	\psi(x) - x = \Omega_{\pm}\Bigl( x\exp\bigl(-(1+\eps)\omega(x)\bigr)\Bigr).
\]
\end{theorem}
These two theorems provide a close to definitive answer to Ingham's question. Note that the Riemann hypothesis (RH) implies that $\psi(x) - x \ll_{\eps} x^{1/2+\eps}$, so that both theorems follow trivially from RH: the first because the conclusion is always true, the second because the premise is always false (at least if $\eta < 1/2$). 

A natural broader context to study Ingham's question is that of \emph{Beurling zeta functions} attached to \emph{Beurling generalized number systems}. Such zeta functions need not satisfy RH, and it is conceivable that a wide range of functions $\eta$ can occur as describing the asymptotic best zero-free region\footnote{Avoiding a precise definition, one can imagine the curve $1-\eta(t) + \I t$ as connecting the extreme points of the convex hull of the set of all zeros. Instead of using a function $\eta$ bounding a zero-free region, it is also possible to state versions of Pintz's theorems which refer to the zeros directly. For this we refer to \cite{Re24}, in particular to Section 8 regarding some subtleties for the converse Pintz theorem.}of some Beurling zeta function. In fact, it is one of the aims of this paper to show this. In contrast, the Riemann zeta function has a fixed zero distribution and the error term in the PNT has a fixed asymptotic behavior. Hence, if $\eta_{\mathrm{RZ}}$ is a function asymptotically describing this zero-distribution, the Pintz theorems for functions $\eta$ which are asymptotically larger than or smaller than $\eta_{\mathrm{RZ}}$ are trivial consequences of the one for $\eta=\eta_{\mathrm{RZ}}$.

Recently, R\'ev\'esz \cite{Re24} generalized Pintz's theorems to the Beurling context (see Theorems \ref{Revesz1} and \ref{Revesz2} below), further confirming that $1-\eta(\abs{t}) \longleftrightarrow x\e^{-\omega_{\eta}(x)}$ is the ``true connection'' between zero-free regions and error terms in the PNT.

\medskip

The aim of this paper is twofold. First, we prove a refinement of the first Pintz--R\'ev\'esz theorem replacing $\eps$ by a function tending to $0$. We note that in the case of the Riemann zeta function, such a refinement was recently obtained by Johnston \cite{Jo24}. Our refinement is slightly better and applies to general Beurling zeta functions, while Johnston's result applies to a wider class of zero-free regions. Nonetheless, our theorem is sufficient to capture any ``reasonable'' choice of $\eta$.

Second, we show that our obtained refinement is sharp, up to the value of the constant of the secondary term $\eps(x)\omega(x)$ in the exponential. We do this by constructing for every ``reasonable'' choice of $\eta$, a Beurling number system $(\MP, \MN)$ whose zeta function $\zeta_{\MP}(s)$ has infinitely many zeros on the contour $\sigma = 1-\eta(\abs{t})$ and none to the right, and for which we also obtain an oscillation result for $\psi_{\MP}(x)-x$, matching the error term predicted by the refined Pintz--R\'ev\'esz theorem (up to the value of the aforementioned constant).

\subsection{Beurling generalized number systems}
Beurling generalized number systems were introduced in \cite{Beurling} in order to investigate the minimal assumptions needed to prove the PNT. They form an abstraction of the multiplicative structure of the primes and integers, and typically lack any additive structure. 

A Beurling generalized number system is a pair of sequences $(\MP, \MN)$ where $\MP$, the sequence of \emph{generalized primes}, consists of real numbers $p_{k}$ satisfying $1<p_{1} \le p_{2} \le \dotso$ and $p_{k}\to\infty$. The sequence of \emph{generalized primes} $\MN$ consists of all possible products of generalized primes $n_{j} = p_{1}^{\nu_{1}}\dotsm p_{k}^{\nu_{k}}$ (including the empty product $n_{0} = 1$) ordered in a non-decreasing fashion. We allow repeated elements in both sequences; if a generalized integer has multiple factorizations into generalized primes, its multiplicity in the sequences $\MN$ equals the number of such distinct factorizations. It can be useful to think of these as different generalized integers with the same numerical value. (Alternatively, one may define a generalized number system as an abstract multiplicative semigroup with a multiplicative norm function.) 

Given a Beurling number system $(\MP, \MN)$, we set for $x\ge1$
\[
	\pi_{\MP}(x) \coloneqq \#\{p_{k} \le x\}, \quad N_{\MP}(x) \coloneqq \#\{n_{j}\le x\}.
\]
Note that $\pi_{\MP}(1)=0$ and $N_{\MP}(1)=1$. In this article, we will assume that
\begin{equation}
\label{theta-well-behaved}
	N_{\MP}(x) = Ax + O(x^{\theta}),
\end{equation}
for some $A>0$ and $\theta \in [0, 1)$. This assumption is referred to as the integers being ``$\theta$-well-behaved''; sometimes it is also called Axiom A. It guarantees that the associated \emph{Beurling zeta function}, defined as
\[
	\zeta_{\MP}(s) \coloneqq \sum_{j\ge0}\frac{1}{n_{j}^{s}} = \prod_{k\ge1}\frac{1}{1-p_{k}^{-s}}
\]
and initially holomorphic on $\Re s > 1$ due to the absolute convergence, is such that $\zeta_{\MP}(s) - \frac{A}{s-1}$ admits analytic continuation to the half-plane $\Re s > \theta$. Finally we define the weighted prime-counting functions
\[
	\Pi_{\MP}(x) \coloneqq \sum_{p_{k}^{\nu}\le x}\frac{1}{\nu}, \quad \psi_{\MP}(x) \coloneqq \sum_{p_{k}^{\nu}\le x}\log p_{k} = \int_{1}^{x}\log u\dif\Pi_{\MP}(u),
\]
which have Mellin--Stieltjes transform
\[
	\int_{1}^{\infty}x^{-s}\dif\Pi_{\MP}(x) = \log\zeta_{\MP}(s), \quad \int_{1}^{\infty}x^{-s}\dif\psi_{\MP}(x) = -\frac{\zeta_{\MP}'(s)}{\zeta_{\MP}(s)}.
\]
More background on Beurling generalized number systems can be found in the book \cite{DZ16}.

We mention that \eqref{theta-well-behaved} implies the existence of an absolute constant $c>0$ such that $\zeta_{\MP}(s)$ has no zeros in the region 
\begin{equation}
\label{dlVP region}
	\sigma > 1 - \frac{c(1-\theta)}{\log\abs{t}}, \quad \abs{t} \ge t_{0}(\MP),
\end{equation}
and consequently
\begin{equation}
\label{dlVP error term}
	\psi_{\MP}(x) - x \ll_{\eps} x\exp\bigl(-(2-\eps)\sqrt{c(1-\theta)\log x}\bigr).
\end{equation}
This follows from Landau's method \cite{La24} (see also Theorem \ref{Hauptsatz} below) and the bound
\begin{equation}
\label{convexity bound}
	\zeta_{\MP}(s) \ll \frac{\abs{t}^{\frac{1-\sigma}{1-\theta}}\log\abs{t}}{\sigma-\theta}, \quad \theta < \sigma \le 1, \quad \abs{t} \ge \e
\end{equation}
in the critical strip (see e.g.\ \cite[Lemma 2.6]{Re21}).
\medskip

In \cite{Re24}, R\'ev\'esz generalized Pintz's theorems to Beurling generalized number systems\footnote{Strictly speaking, Theorem \ref{Revesz1} does not account for a possible zeros $\rho_{0} = \beta_{0}+\I\gamma_{0}$ with $\abs{\gamma_{0}}<1$ and $1-\beta_{0} < \liminf\eta(t)$, as in the definition of $\omega$ we take the infimum only over $t\ge1$. In this paper we assume that $\eta\to0$ and so potential low-lying zeros do not affect the results.}.
\begin{theorem}[R\'ev\'esz]
\label{Revesz1}
Let $(\MP, \MN)$ be a Beurling number system satisfying \eqref{theta-well-behaved}, let $\eta: [0, \infty) \to (0, 1-\theta]$ be an arbitrary real function, and suppose that $\zeta_{\MP}(s)$ has no zeros in the domain $\sigma > 1-\eta(\abs{t})$. Then for every $\eps>0$,
\[
	\psi_{\MP}(x) - x \ll_{\eps} x\exp\bigl(-(1-\eps)\omega(x)\bigr).
\]	 
\end{theorem}

\begin{theorem}[R\'ev\'esz]
\label{Revesz2}
Let $(\MP, \MN)$ be a Beurling number system satisfying \eqref{theta-well-behaved} and let $\eta: [0, \infty) \to (0, 1-\theta]$ be a function such that $f(u)\coloneqq \eta(\e^{u})$ is convex. Assume that $\zeta_{\MP}(s)$ has infinitely many zeros in the domain $\sigma \ge 1-\eta(\abs{t})$. Then for every $\eps>0$, 
\[
	\psi_{\MP}(x) - x = \Omega_{\pm}\Bigl( x\exp\bigl(-(1+\eps)\omega(x)\bigr)\Bigr).
\]
\end{theorem}

\subsection{Statement of the results}
Our first result is a sharpening of the first Pintz--R\'ev\'esz theorem (Theorem \ref{Revesz1}) in the style of Johnston. The proof of his Theorem 2.1 in \cite{Jo24} can be readily adapted to the Beurling setting, and, for our class of functions $\eta$, can be slightly refined\footnote{We refer to Remark \ref{comparison Johnston} below for a comparison with Johnston's result.} still.

Let $\eta: [1, \infty) \to (0, 1-\theta)$ be a function converging to $0$, and write $f(u) = \eta(\e^{u})$. 
In this article, we will always assume that $f$ is eventually $C^{1}$, decreasing, and strictly convex. 
This ensures that for sufficiently large fixed $x$, the function 
\begin{equation}
\label{definition h}
	h(u, x) \coloneqq f(u)\log x + u
\end{equation}
has a minimum attained at a unique point which we denote by $u_{0}(x)$. Note that
\begin{equation}
\label{omega h}
	\omega_{\eta}(x) = \omega(x) = \inf_{u\ge0} h(u,x) = f\bigl(u_{0}(x)\bigr)\log x + u_{0}(x).
\end{equation}
Furthermore, we assume that $f$ is either regularly varying of index $-\alpha$ for some $\alpha \in [0, 1)$, meaning that
\begin{equation}
\label{regular variation}
	\frac{f(\lambda u)}{f(u)} \to \lambda^{-\alpha} \quad \text{as } u \to \infty, \quad \forall \lambda > 0,
\end{equation}
or that $f$ is slowly varying, meaning that
\begin{equation}
\label{slow variation}
	\frac{f(\lambda u)}{f(u)} \to 1\quad \text{as } u \to \infty, \quad \forall \lambda > 0.
\end{equation}
The prototypical examples we have in mind are\footnote{Strictly speaking we should use these definitions only for $u\ge U$ with $U$ sufficiently large so that $f$ is well-defined and $< 1-\theta$, and set e.g.\ $f(u) \coloneqq f(U)$ for $1\le u < U$.} $f(u) = u^{-\alpha}$ and $f(u) = 1/\log u$, for which
\begin{align*}
	u_{0}(x) 	&= (\alpha \log x)^{\frac{1}{1+\alpha}} \quad &&\omega(x) = (1+\alpha)\alpha^{-\frac{\alpha}{1+\alpha}}(\log x)^{\frac{1}{1+\alpha}} , \\
	u_{0}(x) 	&= \frac{\log x}{(\log_{2} x)^{2}}\biggl\{1 + O\biggl(\frac{\log_{3}x}{\log_{2}x}\biggr)\biggr\}
	 \quad  &&\omega(x) = \frac{\log x}{\log_{2} x}\biggl\{1 + O\biggl(\frac{\log_{3}x}{\log_{2}x}\biggr)\biggr\},
\end{align*}
respectively.
All zero-free regions one typically has in mind are either regularly or slowly varying. For example, the function $f$ corresponding to the current asymptotic best zero-free region for the Riemann zeta function, \eqref{Riemann zero-free} below, is regularly varying of index $-2/3$. 

In the case of slow variation, we also require the following technical conditions: 
\begin{gather}
	\frac{1}{f(u)} \log \frac{1}{f(u)} = o(u),  \label{decrease f} \\
	\forall \delta > 0: \limsup_{u\to\infty}\frac{\abs{f'((1+\delta)u)}}{\abs{f'(u)}} < 1. \label{decrease |f'|} 
\end{gather}
These are again fulfilled in most concrete cases. Indeed, when $f$ is slowly varying, $f'$ is often even regularly varying of index $-1$ (as is the case with e.g.\ $f(u) = 1/\log u$), so that the limit above is $1/(1+\delta)$. 
\begin{theorem}
\label{explicit eps}
Let $(\MP, \MN)$ be a Beurling number system satisfying \eqref{theta-well-behaved} and let $f$ be eventually $C^{1}$, decreasing, and strictly convex. Suppose further that $f$ is regularly varying of index $-\alpha$, $0<\alpha\le 1$, or that it is slowly varying and satisfies \eqref{decrease f} and \eqref{decrease |f'|}.
If $\zeta_{\MP}(s)$ has no zeros in the domain $\sigma > 1-f(\log \abs{t})$, then we have for every $\delta>0$,
\[
	\psi_{\MP}(x) - x \ll_{\delta} x\exp\Bigl(-\omega(x) + \bigl(\tfrac{4}{1-\theta}+\delta\bigr)f\bigl(u_{0}(x)\bigr)u_{0}(x)\Bigr)\omega(x)^{9}.
\]
\end{theorem}
For $f(u)=u^{-\alpha}$ and $f(u)=1/\log u$ we obtain
\[
	f\bigl(u_{0}(x)\bigr)u_{0}(x) = (\alpha \log x)^{\frac{1-\alpha}{1+\alpha}} \quad \text{and} \quad f\bigl(u_{0}(x)\bigr)u_{0}(x) = \frac{\log x}{(\log_{2}x)^{3}}\biggl\{1 + O\biggl(\frac{\log_{3}x}{\log_{2}x}\biggr)\biggr\},
\]
respectively. We will later see that always $u_{0}(x) = o(\log x)$, so that $f\bigl(u_{0}(x)\bigr)u_{0}(x) = o\bigl(\omega(x)\bigr)$ in view of \eqref{omega h}. Also, usually $\log\omega(x) = o\bigl(f\bigl(u_{0}(x)\bigr)u_{0}(x)\bigr)$ so that the factor $\omega(x)^{9}$ can be absorbed in $\delta$.

\medskip

Our second result says that the above theorem is close to being sharp.
\begin{theorem}
\label{sharpness theorem}
Let $f$ be an eventually $C^{1}$, decreasing, and strictly convex function satisfying $1/u = o(f(u))$ as $u\to\infty$. Suppose further that $f$ is regularly varying of index $-\alpha$, $0<\alpha\le 1$, or that it is slowly varying and satisfies \eqref{decrease |f'|}. Then there exists a Beurling number system $(\MP, \MN)$ satisfying
\begin{enumerate}
	\item $N_{\MP}(x) = Ax + O_{\eps}(x^{1/2+\eps})$ for some $A>0$ and every $\eps>0$;
	\item for every $\delta>0$,
	\[ 
		\psi_{\MP}(x)-x = \Omega_{\pm}\Bigl(x\exp\bigl(-\omega(x) + (1-\delta)f\bigl(u_{0}(x)\bigr)u_{0}(x)\bigr)\Bigr); 
	\]
	\item $\zeta_{\MP}(s)$ has infinitely many zeros on the contour $\sigma = 1- f(\log\abs{t})$, and none to the right of this contour.	
\end{enumerate}
\end{theorem}
This is a generalization of the seminal result of Diamond, Montgomery, and Vorhauer \cite{DMV}, which states that assuming only \eqref{theta-well-behaved} with some $\theta\ge1/2$, one is not able to improve upon the classical de la Vall\'ee Poussin zero-free region \eqref{dlVP region} and corresponding error term in the PNT \eqref{dlVP error term} for Beurling number systems. The proof of our Theorem \ref{sharpness theorem} heavily relies on their construction.

\subsection{Outline of the paper}
In the next section, we elaborate on our assumptions and prove a couple of technical lemmas concerning the functions $\omega(x)$, $h(u,x)$, and $u_{0}(x)$. Next, in Section \ref{sec: refinement} we prove Theorem \ref{explicit eps}. In Section \ref{sec: DMV function} we define the Diamond--Montgomery--Vorhauer function, which lies at the heart of the construction, and state the necessary properties. The construction demonstrating Theorem \ref{sharpness theorem} is carried out in the subsequent section. Roughly speaking, the idea is to introduce a lot of zeros on (and close to) the contour $\sigma=1-\eta(\abs{t})$ which are arranged such that they exhibit a constructive interference on a particular sequence $(x_{k})_{k}$, leading to the desired oscillation estimate in the PNT. It is also possible to arrange zeros so that they interfere destructively with one another; in particular we also show that $\eps$ cannot be taken to be $0$ in Theorem \ref{Revesz2}, see Proposition \ref{sharpness Revesz2} below. Finally in Section \ref{sec: zero-distribution and growth}, we utilize these examples to shed some light on the connections between zero-free regions for zeta functions and their growth and zero-distribution.

\subsection{Notations}
For functions $f$ and $g$ we write $f = O(g)$ to indicate that $g>0$ and that there exists a constant $C>0$ such that $\abs{f} \le Cg$. We often use Chebyshev's notation $f\ll g$ to mean $f=O(g)$, and $f\gg g$ to mean $g>0$ and $g=O(f)$. Dependence of the constant $C$ on certain parameters will be indicated by a subscript, e.g.\ $f = O_{\eps}(g)$ or $f \ll_{\eps}g$. The statement $f=o(g)$ means that $g>0$ and $\lim f/g = 0$. The point to which the limit is taken is usually clear from the context but is explicitly mentioned otherwise. Occasionally we also write $f \prec g$ to mean $f=o(g)$, and $f\succ g$ to mean $g>0$ and $g=o(f)$. Finally for real $f$ and positive $g$, $f = \Omega_{+}(g)$ means $\limsup f/g > 0$, while $f = \Omega_{-}(g)$ means $\liminf f/g < 0$. If both hold, we write $f=\Omega_{\pm}(g)$ and if at least one holds, we write $f=\Omega(g)$.

We write $\log_{i} x$ for the $i$-times iterated logarithm of $x$.


\section{Some remarks on the assumptions}
The assumption that $f(u)$ is eventually (meaning for sufficiently large $u$) decreasing, $C^{1}$, and strictly convex is not so restrictive. A general $f$ can always be replaced by its \emph{lower convex envelope} $\breve{f}(u) \coloneqq \sup\{F(u): F \text{ convex with } \forall v\, F(v)\le f(v)\}$ without altering the function $\omega(x)$:
\[
	\inf_{u\ge0} \bigl(\breve{f}(u)\log x + u\bigr) = \inf_{u\ge0} \bigl(f(u)\log x+ u\bigr),
\]
see \cite[Lemma 10]{Re24}. Also, for any $f$ converging to $0$, $\breve{f}$ is always (strictly) decreasing, convex, and hence differentiable at all but countably many points. The strict convexity of $f$ and the continuity of $f'$ we require for technical convenience: it guarantees that for every $x\ge1$, $\inf_{u\ge0}\bigl(f(u)\log x + u\bigr)$ is attained at a unique point, namely $u_{0}(x)$. 

Clearly $\breve{f}(u) \le f(u)$, so if $\sigma > 1-f(\log t)$ is a zero-free region, then so is $\sigma>1-\breve{f}(\log t)$. Upon replacing $f$ by $\breve{f}$, the Beurling zeta function $\zeta_{\MP}$ from Theorem \ref{sharpness theorem} will not necessarily have infinitely many zeros \emph{on} the contour $\sigma \ge 1 - f(\log t)$, but it will have infinitely many zeros \emph{on or to the right} of this contour, namely on the contour $\sigma = 1 - \breve{f}(\log t)$.
\medskip

The assumption $f(u)\succ1/u$ in Theorem \ref{sharpness theorem} is also not so restrictive. In view of \eqref{dlVP region}, we may in any case assume that $f(u) \gg 1/u$, while the case $f(u) = c/u$ is the content of the paper \cite{DMV}.
\medskip

The most restrictive assumption we make is that $f$ is of regular or slow variation. This need not be the case in general. Yet, most of the zero-free regions one typically has in mind satisfy this restriction, such as those given by $f(u) = u^{-\alpha}(\log u)^{\beta}$ with $0\le \alpha\le 1$, $\beta\in \R$ ($\beta \le 0$ if $\alpha=0$ or $\alpha=1$). This assumption could potentially be weakened at the cost of more technical proofs. We deemed this not worth pursuing, as this assumption covers all practical cases, and the proofs are already quite technical at certain points.
	
The study of functions of slow and regular variation is the subject of a well-developed theory, see for instance the treatise \cite{BGT}. Karamata's characterisation theorem states that a regularly varying $f$ (positive and measurable) of index $-\alpha$ is of the form $f(u) = u^{-\alpha}L(u)$ for some slowly varying function $L$, and Karamata's representation theorem states that a slowly varying $L$ (positive and measurable) is of the form
\[
	L(u) = \exp\biggl(g(u) + \int_{u_{0}}^{u}\frac{\eps(v)}{v}\dif v\biggr), \quad u\ge u_{0},
\]
for certain bounded measurable functions $g$ and $\eps$, where $g(u)$ converges to a finite limit, and $\eps(u)$ converges to $0$; see e.g.\ \cite[Theorems 1.3.1, 1.4.1]{BGT}. In particular, $f(u) = u^{-\alpha+o(1)}$. Also, the asymptotic relation $L(\lambda u) \sim L(u)$ holds uniformly for $\lambda$ in compact subsets of $(0,\infty)$ (\cite[Theorem 1.2.1]{BGT}).

\medskip

Recall that $h(u, x) = f(u)\log x + u$, so that $u_{0}(x)$ is the unique solution of $f'(u) = -1/\log x$. Indeed, the assumptions on $f$ imply that $f'$ is eventually continuous and strictly increasing with limit $0$, so that this equation has a unique solution. Moreover, for fixed $x$, $h(u,x)$ is decreasing in $u$ until $u = u_{0}(x)$, where $h(u,x)$ attains its minimum $\omega(x)$, after which it is increasing in $u$. Clearly $u_{0}(x)$ increases to $\infty$ as $x\to\infty$.

We mention the following properties of the function $\omega$, taken from \cite[Lemma 12]{Re24}. They in fact hold for arbitrary functions $\eta$ with $\inf \eta = 0$.
 
\noindent For every $\eps>0$ there exists an $x_{0}(\eps)$ so that
	\begin{equation}
		\abs{\omega(x) - \omega(y)} \le \eps\abs{\log(y/x)}, \quad \text{whenever} \quad x > x_{0}(\eps), \quad \sqrt{x} \le y \le x^{2}. \label{omega slow variation}
	\end{equation}
If $1< y < x$,
	\begin{equation}
		\omega(y) < \omega(x) < \frac{\log x}{\log y} \omega(y). \label{inequality omega}
	\end{equation}

The following technical lemma will be useful in the proof of Theorem \ref{sharpness theorem}.
\begin{lemma}
\label{technical}
Suppose that $f$ is either regularly varying, or slowly varying and satisfies \eqref{decrease |f'|}.
For every $\lambda>1$ there exist $\delta>0$ and $x_{0} \ge 1$ so that for $x\ge x_{0}$ and $u\ge \lambda u_{0}(x)$ or $u\le (1/\lambda)u_{0}(x)$,
\[
	h(u, x) - uf(u) \ge \omega(x) + \delta u_{0}(x).
\]
\end{lemma}
\begin{proof}
Suppose first that $f$ is regularly varying of index $-\alpha$. We show that 
\begin{equation}
\label{asymp u-}
	u_{0}(x) \sim \alpha f\bigl(u_{0}(x)\bigr)\log x.
\end{equation}
To prove this, consider $\mu>1$ close to $1$ and $\eps>0$ sufficiently small. Then
\begin{align*}
	f\bigl(u_{0}(x)\bigr)\log x + u_{0}(x) 	&<  f\bigl(\mu u_{0}(x)\bigr)\log x + \mu u_{0}(x)\\
								&\le (\mu^{-\alpha}+\eps) f\bigl(u_{0}(x)\bigr)\log x + \mu u_{0}(x),
\end{align*}
provided that $x\ge x_{0}(\mu, \eps)$ is sufficiently large. From this it follows that 
\[
	u_{0}(x) > \frac{1-\mu^{-\alpha}-\eps}{\mu-1}f\bigl(u_{0}(x)\bigr)\log x, \quad x\ge x_{0}(\mu, \eps).
\]
For every $\eps_{1}>0$, the coefficient of $u_{0}(x)$ above can be made $\ge \alpha -\eps_{1}$ upon choosing $\mu$ sufficiently close to $1$ and $\eps$ sufficiently small. This shows $u_{0}(x) \ge (\alpha + o(1))f\bigl(u_{0}(x)\bigr)\log x$. The other inequality follows analogously, upon considering $\mu <1$ and sufficiently close to $1$.

\medskip

If $u\ge 2(1+1/\alpha)u_{0}(x)$ say, then 
\[
	h(u, x) -uf(u) \ge (1-f(u))u \ge \frac{3(1+1/\alpha)}{2}u_{0}(x) \ge \omega(x) + \frac{1+1/\alpha}{4}u_{0}(x),
\]
provided that $x$ is sufficiently large, as $\omega(x) \sim (1+1/\alpha)u_{0}(x)$ by \eqref{asymp u-}.

If $\lambda u_{0}(x) \le u < 2(1+1/\alpha)u_{0}(x)$, then \eqref{asymp u-} yields
\begin{align*}
	h(u, x) -uf(u) 	&\ge f\bigl(\lambda u_{0}(x)\bigr)\log x + \lambda u_{0}(x) + o(u_{0}(x))\\
				&\ge \Bigl(\frac{\lambda^{-\alpha}}{\alpha} + \lambda +o(1)\Bigr)u_{0}(x) \ge \omega(x) + \Bigl(\frac{\lambda^{-\alpha}}{\alpha}+\lambda-1 -\frac{1}{\alpha}+o(1)\Bigr)u_{0}(x).
\end{align*}
Here, $\lambda^{-\alpha}/\alpha + \lambda -1-1/\alpha > 0$ if $\lambda \neq 1$. Similarly, if $u\le u_{0}(x)/\lambda$,
\[
	h(u,x) - uf(u) \ge \omega(x) + \Bigl(\frac{\lambda^{\alpha}}{\alpha} + \frac{1}{\lambda}-1 - \frac{1}{\alpha} + o(1)\Bigr)u_{0}(x),
\]
and again $\lambda^{\alpha}/\alpha + 1/\lambda - 1 -1/\alpha > 1$ if $\lambda \neq 1$. This proves the lemma if $f$ is regularly varying.

\medskip
Suppose now that $f$ is slowly varying. Then we have
\begin{equation}
\label{u- o}
	u_{0}(x) = o\bigl(f\bigl(u_{0}(x)\bigr)\log x\bigr).
\end{equation}
Indeed, for every $\eps>0$ one has for sufficiently large $x \ge x_{0}(\eps)$:
\[
	f\bigl(u_{0}(x)\bigr)\log x +u_{0}(x) \le f\bigl((1/2)u_{0}(x)\bigr)\log x + u_{0}(x)/2  \le (1+\eps)f\bigl(u_{0}(x)\bigr)\log x + u_{0}(x)/2,
\]
so that $u_{0}(x) \le 2\eps f\bigl(u_{0}(x)\bigr)\log x$.

Let now $\lambda>1$ and write $\lambda = 1+2\nu$, $\nu > 0$. If $u\ge (1+2\nu)u_{0}(x)$, then
\begin{align*}
	&h(u,x) - uf(u) \\	
	&= h\bigl((1+\nu)u_{0}(x), x\bigr) + (1-f(u))u - (1+\nu)u_{0}(x) + \bigl\{f(u)-f\bigl((1+\nu)u_{0}(x)\bigr)\bigr\}\log x \\
					&\ge \omega(x) + (1-f(u))u - (1+\nu)u_{0}(x) - \bigl(u-(1+\nu)u_{0}(x)\bigr)\frac{f'(v)}{f'\bigl(u_{0}(x)\bigr)}.
\end{align*}
Here we used the mean value theorem with some $v\in [(1+\nu)u_{0}(x), u]$ and the fact that $\log x = -1/f'\bigl(u_{0}(x)\bigr)$.
In view of \eqref{decrease |f'|},
\[
	L_{\nu} \coloneqq \limsup_{u\to \infty} \frac{\abs{f'\bigl((1+\nu)u\bigr)}}{\abs{f'(u)}} < 1.
\]
Fix some $0 < \eps < (1-L_{\nu})/2$ As $\abs{f'}$ is decreasing, we have
\[
	\frac{\abs{f'(v)}}{\abs{f'\bigl(u_{0}(x)\bigr)}} \le \frac{\abs{f'\bigl((1+\nu)u_{0}(x)\bigr)}}{\abs{f'\bigl(u_{0}(x)\bigr)}} \le L_{\nu}+\eps,
\]
if $x$ is sufficiently large. Hence
\begin{align*}
	h(u, x) - uf(u) 	&\ge \omega(x) + (1-f(u) - L_{\nu} - \eps)u -(1+\nu)(1-L_{\nu}-\eps)u_{0}(x) \\
					&\ge \omega(x) + \bigl((1-L_{\nu}-\eps)\nu - (1+2\nu)f(u)\bigr)u_{0}(x).
\end{align*}
Now $f \to 0$, so if $x$ is sufficiently large, the above coefficient of $u_{0}(x)$ is $\ge (1-L_{\nu}-2\eps)\nu > 0$.

Writing $1/\lambda = 1-2\nu'$ with $\nu'>0$, a similar reasoning shows that for $u\le (1-2\nu')u_{0}(x)$, 
\[
	h(u, x) - uf(u) \ge \omega(x) + \bigl((1/L - 1 - \eps)\nu' - (1-2\nu')f(u)\bigr)u_{0}(x),
\]
where 
\[
	\frac{1}{L} = \frac{1}{L_{\frac{1}{1-\nu'}-1}} = \liminf_{u\to\infty} \frac{\abs{f'\bigl((1-\nu')u\bigr)}}{\abs{f'(u)}} > 1.
\]
This concludes the proof.
\end{proof}
At this point it is worth mentioning that \eqref{asymp u-} and \eqref{u- o} imply that $u_{0}(x) = o(\log x)$, and consequently also $\omega=o(\log x)$. On the other hand, as we may assume that $f(u)\gg1/u$ in view of \eqref{dlVP region}, $u_{0}(x) \gg \sqrt{\log x}$ and $\omega(x)\gg\sqrt{\log x}$.

Finally, although we do not need it in our proofs, it is interesting to state the following fact.
\begin{proposition}
\label{omega regularly varying}
If $f$ is regularly varying of index $-\alpha$, then $u_{0}(x)$ and $\omega(x)$ are regularly varying functions of $\log x$ of index $\frac{1}{1+\alpha}$. If $f$ is slowly varying, then $\omega(x)$ is a regularly varying function of $\log x$ of index $1$.
\end{proposition}

\begin{proof}
Suppose first that $f$ is regularly varying of index $-\alpha$. Let $\lambda>1$. We will show that $u_{0}(x^{\lambda}) \sim \lambda^{\frac{1}{1+\alpha}}u_{0}(x)$. In view of \eqref{asymp u-} this implies that also $\omega(x^{\lambda}) \sim \lambda^{\frac{1}{1+\alpha}}\omega(x)$.

Combining \eqref{inequality omega} with \eqref{asymp u-} we obtain
\[
	\Bigl(1+\frac{1}{\alpha} + o(1)\Bigr)u_{0}(x) < \Bigl(1+\frac{1}{\alpha} + o(1)\Bigr)u_{0}(x^{\lambda}) < \lambda\Bigl(1+\frac{1}{\alpha} + o(1)\Bigr)u_{0}(x).
\]
Hence we may write $u_{0}(x^{\lambda}) = c_{\lambda}(x)u_{0}(x)$, with $c_{\lambda}(x) \asymp 1$. As \eqref{regular variation} holds uniformly for $\lambda$ in compacta, it follows that 
\[
	\omega(x^{\lambda}) = \bigl(c_{\lambda}(x)^{-\alpha}+o(1)\bigr) f\bigl(u_{0}(x)\bigr)\lambda\log x + c_{\lambda}(x)u_{0}(x) = \Bigl(\frac{\lambda c_{\lambda}(x)^{-\alpha}}{\alpha} + c_{\lambda}(x) + o(1)\Bigr)u_{0}(x).
\]
On the other hand,
\[
	\omega(x^{\lambda}) \le h\bigl(\lambda^{\frac{1}{1+\alpha}}u_{0}(x), \log x^{\lambda}\bigr) = \lambda^{\frac{1}{1+\alpha}}\Bigl(\frac{1}{\alpha} + 1 + o(1)\Bigr)u_{0}(x).
\]
This now implies that $c_{\lambda}(x) \to \lambda^{\frac{1}{1+\alpha}}$. Indeed, the function $g(c) = \lambda c^{-\alpha}/\alpha + c$ attains its minimum $\lambda^{\frac{1}{1+\alpha}}(1+1/\alpha)$ at $c=\lambda^{\frac{1}{1+\alpha}}$. Hence, if $c_{\lambda}(x)$ is an $\eps$ away from $ \lambda^{\frac{1}{1+\alpha}}$ for arbitrarily large values of $x$, then $g\bigl(c_{\lambda}(x)\bigr) \ge \lambda^{\frac{1}{1+\alpha}}(1/\alpha+1) +\eps'$ for some $\eps'$, for arbitrarily large $x$. For those $x$ we would obtain
\[
	\Bigl(\lambda^{\frac{1}{1+\alpha}}\Bigl(\frac{1}{\alpha}+1\Bigr) + \eps' +o(1)\Bigr)u_{0}(x) \le \lambda^{\frac{1}{1+\alpha}}\Bigl(\frac{1}{\alpha} + 1 + o(1)\Bigr)u_{0}(x),
\]
a contradiction. Thus, $c_{\lambda}(x) \to \lambda^{\frac{1}{1+\alpha}}$ and $u_{0}(x^{\lambda}) = \bigl(\lambda^{\frac{1}{1+\alpha}}+o(1)\bigr)u_{0}(x)$.

\medskip
Suppose now that $f$ is slowly varying and let again $\lambda > 1$. Combining \eqref{inequality omega} and \eqref{u- o} we obtain
\[
	\bigl(1+o(1)\bigr)f\bigl(u_{0}(x^{\lambda})\bigr)\log x^{\lambda} < \lambda\bigl(1+o(1)\bigr)f\bigl(u_{0}(x)\bigr)\log x,
\]
so that $f\bigl(u_{0}(x^{\lambda})\bigr) \le \bigl(1+o(1)\bigr)f\bigl(u_{0}(x)\bigr)$. On the other hand, by the mean value theorem, the fact that $f'$ is increasing, and 
\[
	u_{0}(x^{\lambda}) = o\bigl(f\bigl(u_{0}(x^{\lambda})\bigr)\log x^{\lambda}\bigr)  = o\bigl(f\bigl(u_{0}(x)\bigr)\log x\bigr),
\]
we see that 
\begin{align*}
	f\bigl(u_{0}(x^{\lambda})\bigr) 	&\ge f\bigl(u_{0}(x)\bigr) + \bigl(u_{0}(x^{\lambda}) - u_{0}(x)\bigr)f'\bigl(u_{0}(x)\bigr) \\
							&=  f\bigl(u_{0}(x)\bigr) - \frac{u_{0}(x^{\lambda}) - u_{0}(x)}{\log x} = \bigl(1-o(1)\bigr)f\bigl(u_{0}(x)\bigr).
\end{align*}
We conclude that $f\bigl(u_{0}(x^{\lambda})\bigr) \sim f\bigl(u_{0}(x)\bigr)$ and 
\[
	\omega(x^{\lambda}) \sim f\bigl(u_{0}(x^{\lambda})\bigr)\log x^{\lambda} \sim \lambda f\bigl(u_{0}(x)\bigr)\log x \sim \lambda \omega(x).
\]
\end{proof}
As a consequence of this proposition we have (with $\alpha=0$ if $f$ is slowly varying) 
\[
	\omega(x) = (\log x)^{\frac{1}{1+\alpha}+o(1)}.
\]

Most practically occurring slowly varying functions $f$ (e.g.\ $f(u)=1/\log u$) are such that $f'$ is regularly varying of index $-1$. One may verify that in that case, the inverse function of the derivative, $(f')^{-1}$, is regularly varying of index $1/(-1) = -1$ at the origin. So in that case,
\[
	u_{0}(x^{\lambda}) = (f')^{-1}\Bigl(-\frac{1}{\log x^{\lambda}}\Bigr) = (f')^{-1}\biggl(\frac{1}{\lambda} f'\bigl(u_{0}(x)\bigr)\biggr) \sim \lambda u_{0}(x),
\]
so that $u_{0}(x)$ is a regularly varying function of $\log x$ of index $1$. 

\section{Refinement of the Johnston--Pintz--R\'ev\'esz theorem}
\label{sec: refinement} 
The starting point in the original proof of Pintz is an explicit formula for the error term in the prime number theorem, expressing it as a sum over non-trivial zeros of the Riemann zeta function. For  Beurling number systems satisfying \eqref{theta-well-behaved}, the following explicit Riemann--von Mangoldt-style formula for $\psi_{\MP}(x)$ was established by R\'ev\'esz (see\footnote{In the two quoted results, the sum is over the zeros lying to the right of a certain broken line $\Gamma_{b}$ in the strip $(\theta+b)/2 \le \sigma \le b$ and $T$ is restricted to belong to a certain sequence $(t_{k})_{k}$ with $1\le t_{k+1}-t_{k} \le 2$. Using the fact that the number of zeros in $[(\theta+b)/2, 1]\times \I[t, t+1]$ is $O_{b}(\log t)$, one readily sees that the contribution from the zeros with $(\theta+b)/2 \le \beta< b$ can be absorbed in the error term and that the formula holds for all $T$ and not only on the sequence $t_k$ defined in R\'ev\'esz's paper.} \cite[Theorem 5.1]{Re21} and the corrected version \cite[Proposition 1]{Re24}). Let $b \in (\theta,1)$ and $4\le T\le x$. Then
\begin{equation}
\label{explicit Riemann--von Mangoldt}
	\psi_{\MP}(x) = x - \sum_{\substack{\beta\ge b,\\ \abs{\gamma}\le T}}\frac{x^{\rho}}{\rho} + O_{b}\Bigl\{\Bigl(x^{b} + \frac{x}{T}\Bigr)(\log x)^{3}\Bigr\},
\end{equation}
where the sum is over the zeros $\rho=\beta+\I\gamma$ of $\zeta_{\MP}(s)$.

\medskip

The second main ingredient is a zero-density theorem for Beurling zeta functions. Letting $N(\sigma, T)$ be the number of zeros of $\zeta_{\MP}$ in the rectangle $[\sigma, 1]\times\I[-T,T]$ (counted with multiplicity), we consider estimates of the following form:
\begin{equation}
\label{general zero-density}
	N(\sigma, T) \ll T^{A(1-\sigma)^{B}}(\log T)^{C}, \quad \text{uniformly for } \sigma_{0} \le \sigma \le 1,
\end{equation}
for certain parameters $\sigma_{0} \in (\theta, 1)$, $A>0$, $B\ge1$, $C\ge0$.

The current record for Beurling zeta functions satisfying \eqref{theta-well-behaved} is (see \cite[Theorem 1.2]{B25})
\begin{equation}
\label{zero-density}
	N(\sigma, T) \ll T^{\frac{4(1-\sigma)}{1-\theta}}(\log T)^{9}, \quad \text{uniformly for }  \frac{1+\theta}{2} \le \sigma \le 1,
\end{equation}
while in \cite{DMV} one finds for every $\eps>0$ and $\theta\in (1/2, 1)$ an example exhibiting $N(\sigma, T_{n}) > T_{n}^{\frac{(1-\eps)(1-\sigma)}{1-\theta}}$ on a certain unbounded sequence $(T_{n})_{n}$; see also \eqref{number of zeros} below.

Theorem \ref{explicit eps} follows immediately from the following more general theorem with the choice $A=\frac{4}{1-\theta}$, $B=1$, $C=9$.
\begin{theorem}
\label{refinement JPR}
Let $(\MP, \MN)$ be a Beurling number system satisfying \eqref{theta-well-behaved} and with $\zeta_{\MP}(s)$ satisfying \eqref{general zero-density}. Let $f$ be eventually $C^{1}$, decreasing, strictly convex, and suppose further that $f$ is regularly varying of index $-\alpha$, $0<\alpha\le 1$, or that it is slowly varying and satisfies $\frac{1}{f(u)^{B}}\log\frac{1}{f(u)} = o(u)$ and \eqref{decrease |f'|}.

If $\zeta_{\MP}(s)$ has no zeros in the domain $\sigma > 1-f(\log \abs{t})$, then we have for every $\delta>0$,
\[
	\psi_{\MP}(x) - x \ll_{\delta} x\exp\Bigl(-\omega(x) + (A+\delta)f(u_{0}(x))^{B}u_{0}(x)\Bigr)\omega(x)^{C}.
\]
\end{theorem}
\begin{remark}
\label{comparison Johnston}
\begin{itemize}
\item If $f(u) \succ (\log u/u)^{1/B}$, then the factor $\omega(x)^{C}$ can be omitted, as then $\log \omega (x) \asymp \log_{2} x \prec f(u_{0}(x))^{B}u_{0}(x)$ since $\log u_{0}(x) \asymp \log_{2}x$.
\item For the Riemann zeta function, Johnston \cite[Theorem 2.1]{Jo24} obtained the above result with $(A+\delta)f(u_{0}(x))^{B}u_{0}(x)$ replaced by $2A\omega(x)(\frac{\omega(x)}{\log x})^{B}$, without needing to assume however that $f$ is regularly or slowly varying. Note that, as $\omega(x) = f(u_{0}(x))\log x + u_{0}(x)$,
\[
	(A+\delta)f(u_{0}(x))^{B}u_{0}(x) < (A+\delta)\Bigl(\frac{\omega(x)}{\log x}\Bigr)^{B}\omega(x),
\]
and that if $f$ is slowly varying, 
\[
	(A+\delta)f(u_{0}(x))^{B}u_{0}(x) = o\biggl(\Bigl(\frac{\omega(x)}{\log x}\Bigr)^{B}\omega(x)\biggr),
\]
in view of \eqref{u- o}. For example, when $f(u) = u^{-\alpha}$ or $f(u)=1/\log u$, we have
\begin{align*}
	f\bigl(u_{0}(x)\bigr)^{B}u_{0}(x)	&= (\alpha\log x)^{\frac{1-B\alpha}{1+\alpha}} 
	\quad &&\omega(x)\Bigl(\frac{\omega(x)}{\log x}\Bigr)^{B} = \frac{(1+\alpha)^{1+B}}{\alpha}(\alpha\log x)^{\frac{1-B\alpha}{1+\alpha}}, \quad \text{and}\\
	f\bigl(u_{0}(x)\bigr)^{B}u_{0}(x)	&\sim\frac{\log x}{(\log_{2} x)^{2+B}} \quad &&\omega(x)\Bigl(\frac{\omega(x)}{\log x}\Bigr)^{B} \sim \frac{\log x}{(\log_{2}x)^{1+B}},
\end{align*}
respectively.
\end{itemize}
\end{remark}

\begin{proof}[Proof of Theorem \ref{refinement JPR}]
We first fix some small $0 <\eps<1-\sigma_{0}$, $0<\delta <1$, and $0<\delta'<1$. We apply \eqref{explicit Riemann--von Mangoldt} with $b=1-\eps$ and $T = \e^{2\omega(x)}$. Clearly $T<x$ since $\omega(x) = o(\log x)$. As $1-\beta \ge f(\log \gamma)$ for every zero $\rho=\beta+\I\gamma$, $\gamma\ge\e$ we obtain
\[
	\frac{\psi_{\MP}(x)-x}{x} \ll \sum_{\substack{\beta\ge 1-\eps,\\ \e\le \gamma \le T}}\exp\bigl(-(f(\log \gamma)\log x + \log \gamma)\bigr) + \bigl(x^{-\eps}+\e^{-2\omega(x)}\bigr)(\log x)^{3}.
\]
We first remark that $x^{-\eps}(\log x)^{3} \ll_{\eps} \e^{-\omega(x)}$ and $\e^{-2\omega(x)}(\log x)^{3} \ll \e^{-\omega(x)}$, as we may assume that $f(u)\gg 1/u$ and so $\omega(x) \gg \sqrt{\log x}$ in view of \eqref{dlVP region}.

Next we define $u_{\delta, +}(x)$ and $u_{\delta,-}(x)$ to be the unique solutions of $h(u, x) = (1+\delta)\omega(x)$ with $u> u_{0}(x)$ and $u < u_{0}(x)$ respectively. Note that $u_{\delta,+}(x) \le (1+\delta)\omega(x)$ as $h\bigl((1+\delta)\omega(x), x\bigr) > (1+\delta)\omega(x)$. We now set
\begin{gather*}
	\tau = \exp(u_{0}(x)), \quad \tau_{+} = \exp(u_{\delta,+}(x)), \quad \tau_{-} = \exp(u_{\delta,-}(x)), \\
	\sum_{\substack{\beta\ge 1-\eps,\\ \e\le \gamma \le T}}\exp\bigl(-h(\log \gamma, x)\bigr) = S_{1} + S_{2} + S_{3},
\end{gather*}
where $S_{1}$, $S_{2}$, and $S_{3}$ are the sums over the zeros in the ranges $\gamma \le \tau_{-}$, $\tau_{-}< \gamma \le \tau_{+}$, and $\tau_{+}<\gamma\le T$ respectively.

Employing \eqref{zero-density} we obtain
\[
	S_{1} \ll_{\eps} (\tau_{-})^{\frac{5\eps}{1-\theta}}\exp\bigl(-(1+\delta)\omega(x)\bigr) \le \exp\Bigl(-\Bigl(1+\delta-\frac{10\eps}{1-\theta}\Bigr)\omega(x)\Bigr),
\]
where we used that $h(u,x)$ is decreasing in $u$ for $u< u_{0}(x)$, and $u_{\delta,-}(x) \le h\bigl(u_{\delta,-}(x), x\bigr) \le 2\omega(x)$. Hence if $\eps$ is sufficiently small with respect to $\delta$, this contribution is $O_{\eps}(\e^{-\omega(x)})$.

To estimate $S_{3}$ we perform a dyadic splitting. Again using \eqref{zero-density} and the fact that $h(u,x)$ is now increasing in $u$ for $u>u_{0}(x)$, we see that 
\begin{align*}
	S_{3}&\ll \sum_{j=0}^{\log x}(2^{j}\tau_{+})^{\frac{5\eps}{1-\theta}}\exp\Bigl(-h\bigl(\log(2^{j}\tau_{+}), x\bigr)\Bigr) \\
		&\le \sum_{j=0}^{\log x}(2^{j}\tau_{+})^{-\frac{\eps}{1-\theta}}\exp\Bigl(-\Bigl(1-\frac{6\eps}{1-\theta}\Bigr)h\bigl(\log(2^{j}\tau_{+}), x\bigr)\Bigr)\\
		&\le \exp\Bigl(-\Bigl(1-\frac{6\eps}{1-\theta}\Bigr)(1+\delta)\omega(x)\Bigr)\sum_{j=0}^{\infty}2^{-\eps j} \ll_{\eps} \e^{-\omega(x)},
\end{align*}
provided that $\eps$ is sufficiently small with respect to $\delta$.

\medskip

Next we write $S_{2}$ as $S_{2} = S_{2,a} + S_{2,b}$, where $S_{2,a}$ is the sum over the zeros with $\tau_{-} < \gamma \le \tau_{+}$ and $1-\eps \le \beta \le 1-(1+\delta')f\bigl(u_{0}(x)\bigr)$, and $S_{2,b}$ is the sum over the zeros with $\tau_{-} < \gamma \le \tau_{+}$ and $1-(1+\delta')f(u_{0}(x)) < \beta \le 1-f(\log \gamma)$.

Suppose first that $f$ is slowly varying of index $-\alpha$, $0<\alpha\le 1$. Then \eqref{asymp u-} implies that 
\[
	\omega(x) \sim \biggl(1+\frac{1}{\alpha}\biggr)u_{0}(x), \quad h(\lambda u_{0}(x), x) \sim \frac{\lambda^{-\alpha}+\alpha\lambda}{1+\alpha}\omega(x),
\]
so that $u_{\delta, \pm}(x) \sim \lambda_{\delta, \pm}u_{0}(x)$, where $\lambda_{\delta, +}$ (resp.\ $\lambda_{\delta, -}$) is the unique solution of $\frac{\lambda^{-\alpha}+\alpha\lambda}{1+\alpha} = 1+\delta$ with $\lambda > 1$ (resp.\ with $\lambda< 1$). Furthermore, $\lambda_{\delta, \pm} = 1 + O_{\alpha}(\sqrt{\delta})$. Hence we obtain
\begin{align*}
	S_{2,a} 	&\ll_{\eps}(\tau_{+})^{\frac{5\eps}{1-\theta}}\exp\Bigl(-(1+\delta')f\bigl(u_{0}(x)\bigr)\log x - \log \tau_{-}\Bigr)\\
			&\ll \exp\Bigl(-\omega(x) - \delta'f\bigl(u_{0}(x)\bigr)\log x + u_{0}(x) - (\lambda_{\delta,-}-\eps)u_{0}(x) + \frac{5\eps}{1-\theta}(\lambda_{\delta,+}+\eps)u_{0}(x)\Bigr) \\
			&\ll \exp\Bigl(-\omega(x) - \delta'f\bigl(u_{0}(x)\bigr)\log x +\Bigl(\eps + \frac{10\eps}{1-\theta}+O(\sqrt{\delta}\,)\Bigr)u_{0}(x)\Bigr).
\end{align*}
Using \eqref{asymp u-}, this is $\le \e^{-\omega(x)}$ provided that $\eps$ and $\delta$ are sufficiently small with respect to $\delta'$. 

For the final sum $S_{2,b}$ we employ \eqref{general zero-density} to obtain
\begin{align*}
	S_{2,b}	&\ll_{\eps}(\tau_{+})^{A[(1+\delta')f(u_{0}(x))]^{B}}(\log\tau_{+})^{C}\e^{-\omega(x)}\\
			&\ll\exp\Bigl(-\omega(x) + A(1+\delta')^{B}\bigl(1+O(\sqrt{\delta}\,)\bigr)f\bigl(u_{0}(x)\bigr)^{B}u_{0}(x)\Bigr)\omega(x)^{C}.
\end{align*}

\medskip

Suppose now that $f$ is slowly varying. Then \eqref{u- o} implies $f\bigl(u_{0}(x)\bigr)\log x \sim \omega(x)$, so that by a dyadic splitting,
\begin{align*}
	S_{2,a}	&\ll_{\eps}\sum_{j=0}^{\log x}(2^{j}\tau_{-})^{\frac{5\eps}{1-\theta}-1}\exp\bigl(-(1+\delta')f\bigl(u_{0}(x)\bigr)\log x\bigr) \\
			&\ll \sum_{j=0}^{\log x}2^{-j/2} \e^{-\omega(x)} \ll \e^{-\omega(x)}.
\end{align*}

In the sum $S_{2,b}$ we make a final subdivision $S_{2,b} = S_{2,b,i} + S_{2, b,ii}$ based on whether $\tau_{-} < \gamma \le \tau^{1+\delta}$  or $\tau^{1+\delta} < \gamma \le \tau_{+}$ (note that $u_{0}(x) = o(u_{\delta, +}(x))$ in this case). Using \eqref{general zero-density} and $\log \tau < \omega(x)$ we obtain
\[
	S_{2,b,i} \ll (\tau^{1+\delta})^{A[(1+\delta')f(u_{0}(x))]^{B}}\omega(x)^{C}\e^{-\omega(x)} = \exp\Bigl(-\omega(x) + A(1+\delta')^{B}(1+\delta)f\bigl(u_{0}(x)\bigr)^{B}u_{0}(x)\Bigr)\omega(x)^{C}.
\]
To estimate $S_{2,b,ii}$ we perform again a dyadic splitting to see that
\begin{align*}
	S_{2,b,ii}	&\ll \sum_{j=0}^{\log x}(2^{j}\tau^{1+\delta})^{A[(1+\delta')f(u_{0}(x))]^{B}}\omega(x)^{C}\exp\Bigl(-f\bigl(\log(2^{j}\tau^{1+\delta})\bigr)\log x - \log(2^{j}\tau^{1+\delta})\Bigr) \\
			&\ll \sum_{j=0}^{\log x}(2^{j}\tau^{1+\delta})^{-\delta f(u_{0}(x))^{B}}\exp\Bigl(-f\bigl(\log(2^{j}\tau^{1+\delta})\bigr)\log x -(1-\epsilon) \log(2^{j}\tau^{1+\delta})\Bigr)\omega(x)^{C},
\end{align*}
where we have written $\epsilon = [A(1+\delta')^{B}+\delta]f\bigl(u_{0}(x)\bigr)^{B}$.
Consider now the function
\[
	f(u)\log x + (1-\epsilon) u.
\]
We claim that it is increasing for $u\ge (1+\delta)u_{0}(x)$ if $f$ satisfies \eqref{decrease |f'|}. Indeed, noting that $f'\bigl(u_{0}(x)\bigr) = -1/\log x$ and $\epsilon = \epsilon(x) = o(1)$,
\begin{align*}
	\frac{\partial}{\partial u}\bigl(f(u)\log x + (1-\epsilon) u\bigr) 	&= f'(u)\log x + 1-\epsilon \ge f'\bigl((1+\delta)u_{0}(x)\bigr)\log x + 1-\epsilon \\
													&= 1 - \epsilon - \frac{\abs{f'\bigl((1+\delta)u_{0}(x)\bigr)}}{\abs{f'\bigl(u_{0}(x)\bigr)}} \ge 0,
\end{align*}
if $x$ is sufficiently large. Hence
\[
	S_{2,b,ii} \ll \sum_{j=0}^{\log x}(2^{j}\tau^{1+\delta})^{-\delta f(u_{0}(x))^{B}} \exp\Bigl(-f\bigl((1+\delta)u_{0}(x)\bigr)\log x - (1-\epsilon)(1+\delta)u_{0}(x)\Bigr)\omega(x)^{C}.
\]
Using that $\frac{1}{f(u)^{B}}\log\frac{1}{f(u)} = o(u)$ we obtain
\[
	\sum_{j=0}^{\log x}(2^{j}\tau^{1+\delta})^{-\delta f(u_{0}(x))^{B}} \ll \frac{1}{\delta f\bigl(u_{0}(x)\bigr)^{B}}\exp\bigl(-\delta f\bigl(u_{0}(x)\bigr)^{B}(1+\delta)u_{0}(x)\bigr) \ll_{\delta} 1.
\]
Next, from the mean value theorem and the fact that $f'$ is increasing, we see that 
\[
	f\bigl((1+\delta)u_{0}(x)\bigr)\log x \ge f\bigl(u_{0}(x)\bigr)\log x + \delta u_{0}(x)f'\bigl(u_{0}(x)\bigr)\log x = f\bigl(u_{0}(x)\bigr)\log x - \delta u_{0}(x),
\]
whence
\begin{align*}
	S_{2, b, ii} 	&\ll_{\delta} \exp\Bigl(-\omega(x) + \epsilon(1+\delta)u_{0}(x)\Bigr)\omega(x)^{C} \\
				&= \exp\Bigl(-\omega(x) + (1+\delta)[A(1+\delta')^{B}+\delta]f\bigl(u_{0}(x)\bigr)^{B}u_{0}(x)\Bigr)\omega(x)^{C}.
\end{align*}

If $\tilde{\delta} > 0$ is a fixed positive number, we conclude that upon choosing $\eps$, $\delta$, and $\delta'$ sufficiently small, 
\[
	\psi_{\MP}(x)-x \ll_{\eps, \delta, \delta'} x\exp\Bigl(-\omega(x) + (A+\tilde{\delta})f\bigl(u_{0}(x)\bigr)^{B}u_{0}(x)\Bigr)\omega(x)^{C}.
\]
This concludes the proof.
\end{proof}

For the Riemann zeta function, the asymptotically best known zero-free region and zero-density estimate near the $1$-line take the shape
\begin{gather}
	\sigma > 1-\eta(\abs{t}) = 1 - \frac{c}{(\log \abs{t})^{2/3}(\log_{2} \abs{t})^{1/3}}, \quad \abs{t} \ge t_{0} \label{Riemann zero-free}\\
	N(\sigma, T) \ll T^{A(1-\sigma)^{3/2}}(\log T)^{C}. \label{Riemann zero-density}
\end{gather}
The main tool in proving these estimates is the Vinogradov--Korobov estimate $\zeta(s) \ll \abs{t}^{a(1-\sigma)^{3/2}}(\log\abs{t})^{2/3}$ for the Riemann zeta function near the $1$-line. At the time of writing this paper, the best known constants in \eqref{Riemann zero-free} and \eqref{Riemann zero-density} seem to be (see \cite[Theorem 1.3]{Be24} and \cite[Theorem 1.1]{Be23} respectively)
\[
	c = \frac{1}{48.0718}, \quad A = 57.8875, \text{ with } C\approx 10.95.
\] 
The zero-free region \eqref{Riemann zero-free} yields
\[
	\omega(x) \asymp u_{0}(x) \asymp \frac{(\log x)^{3/5}}{(\log_{2} x)^{1/5}},
\]
so that, with $B=3/2$, $f\bigl(u_{0}(x)\bigr)^{B}u_{0}(x) \asymp (\log_{2} x)^{-1/2}\ll 1$. Hence in this situation, the effect of the constant $C$ is more important than that of the constant $A$. As far as the author is aware, the current lowest value of $C$ is due to Pintz (\cite[Theorem 1]{Pi22}), who showed $N(\sigma, T) \ll T^{86(1-\sigma)^{3/2}}(\log T)^{6}$, yielding
\[
	\psi(x)-x \ll x\e^{-\omega_{\eta}(x)}\frac{(\log x)^{18/5}}{(\log_{2} x)^{6/5}}, \quad \text{with $\eta$ as in \eqref{Riemann zero-free}}.
\]
It is an interesting question whether one can obtain a ``log-free $B=3/2$-type'' zero-density estimate. The original $B=3/2$-type estimate is due to Hal\'asz and Tur\'an (\cite[Theorem 3]{HT69}), who showed that 
\[
	N(\sigma, T) \ll T^{(1-\sigma)^{3/2}(\log\frac{1}{1-\sigma})^{3}}.
\]
They said ``We laid no stress on obtaining the smallest possible exponent of $\log 1/(1-\sigma)$''. Interestingly, if one could lower this exponent $D$ of $\log 1/(1-\sigma)$ from $D=3$ to any value below $3/2$, it would outperform in this application any estimate of the form \eqref{Riemann zero-density} with $B=3/2$, $C>0$, yielding (along the same lines as the proof of Theorem \ref{refinement JPR})
\[
	\psi(x)-x \ll x\e^{-\omega_{\eta}(x)}\exp\bigl(O((\log_{2} x)^{D-1/2})\bigr).
\]
Finally, note that
\[	
	N(\sigma, T) \ll T^{(1-\sigma)^{3/2}(\log\frac{1}{1-\sigma})^{D}} \implies N\biggl(1-\frac{1}{(\log T)^{2/3}(\log_{2} T)^{2D/3}}, T\biggr) \ll 1,
\]
so that $D<1/2$ would improve the Vinogradov--Korobov zero-free region \eqref{Riemann zero-free} asymptotically. Therefore, a mere optimization of the Hal\'asz--Tur\'an argument will likely not yield improvements beyond $D=1/2$.


\section{The Diamond--Montgomery--Vorhauer function}
\label{sec: DMV function}
We now turn our attention to the construction of Beurling zeta functions having infinitely many zeros on a prescribed contour. A common strategy for constructing Beurling number systems satisfying certain properties, is to first construct a system \emph{in the extended sense} satisfying those properties, and to then approximate it by an actual (discrete) system. By a Beurling number system in the extended sense we mean a pair of right-continuous non-decreasing functions $\MB = (\Pi(x), N(x))$, supported on $[1,\infty)$, and satisfying
\begin{equation}
\label{theorem arithmetic}
	\Pi(1) = 0, \quad N(1)=1, \quad \dif N = \exp^{\ast}(\dif\Pi) \coloneqq \sum_{n=0}^{\infty}\frac{\dif\Pi^{\ast n}}{n!}.
\end{equation}
Here $\dif \Pi^{\ast n}$ denotes the $n$-th convolution power with respect to multiplicative convolution:
\[
	(\dif F\ast \dif G)(E) \coloneqq \iint_{uv\in E}\dif F(u)\dif G(v),
\]
and $\dif F^{\ast 0} = \delta_{1}$, the unit point mass at $x=1$, which is the unit for convolution.
If $\dif F = \sum_{n}f(n)\delta_{n}$, $\dif G(x) = \sum_{n}g(n)\delta_{n}$, then $\dif F\ast\dif G = \sum_{n}h(n)\delta_{n}$, where $h = f\ast g$ is the usual Dirichlet convolution. 

The last equation in \eqref{theorem arithmetic} should be thought of as the ``fundamental theorem of arithmetic'' for the system $\MB$, connecting the ``primes'' with the ``integers''. On the Mellin transform side this translates to the ``Euler product formula''
\[
	\zeta_{\MB}(s) \coloneqq \int_{1^{-}}^{\infty}x^{-s}\dif N(x) = \exp\biggl(\int_{1}^{\infty}x^{-s}\dif\Pi(x)\biggr).
\]
We also define the Chebyshev function of the system as $\psi_{\MB}(x) \coloneqq \int_{1}^{x}\log u\dif\Pi(u)$, and note that $\int_{1}^{\infty}x^{-s}\dif\psi_{\MB}(x) = -\frac{\zeta_{\MB}'(s)}{\zeta_{\MB}(s)}$.
We refer to \cite[Chapters 2--3]{DZ16} for more background on these notions. A simple but fundamental example of a system in the extended sense is given by
\[
	\Pi(x) = \Li(x) \coloneqq \int_{1}^{x}\frac{1 - u^{-1}}{\log u}\dif u, \quad N(x) = x, \quad (x\ge1).
\]
Its associated Chebyshev and zeta functions are $\psi_{\MB}(x) = x - \log x -1$ and $\zeta_{\MB}(s) = \frac{s}{s-1}$.

One way to introduce zeros to the zeta function is by considering, for fixed $\rho \in \C$, the function
\[
	\Li(x^{\rho}) \coloneqq \int_{1}^{x}\frac{u^{\rho-1}-u^{-1}}{\log u}\dif u.
\]
One may check that 
\[
	\int_{1}^{\infty}x^{-s}\dif\Li(x^{\rho}) = \log\frac{s}{s-\rho} \quad \text{for } \Re s > \Re \rho
\]
($\log$ denoting the principal branch of the logarithm), so that subtracting $\Li(x^{\rho})$ from $\Pi(x)$ introduces a zero at $s=\rho$ for the zeta function. Note also that 
\[
	\Li(x^{\rho}) \sim \frac{x^{\rho}}{\rho \log x}, \quad \int_{1}^{x}\log u \dif\Li(u^{\rho}) = \frac{x^{\rho}-1}{\rho}-\log x \sim \frac{x^{\rho}}{\rho}, \quad x\to\infty,
\]
which aligns with the idea that a zeta zero at $s=\rho$ should roughly correspond to a perturbation to $\psi_{\MB}(x)$ of the form $-\frac{x^{\rho}}{\rho}$ via the Riemann--von Mangoldt explicit formula \eqref{explicit Riemann--von Mangoldt}. The terms $\Li(x^{\rho})$ were used in \cite{BDR25} to construct Beurling zeta functions having zeros and poles precisely at prescribed \emph{finite} sets. This construction seems not feasible if one wants to have \emph{infinitely many} zeros. The issue is guaranteeing that $\Pi(x)$, or equivalently, $\psi_{\MB}(x)$, is non-decreasing. Indeed, an infinite sequence of zeros $(\rho_{k})_{k}$ introduced in this way contributes to the derivative $\psi_{\MB}'(x)$ as $-\sum_{k}(x^{\rho_{k}-1}-x^{-1})$, and it is unclear how to handle this in general divergent series\footnote{For the ordinary primes one has the exact formula
\[
	\frac{\psi(x^{+})+\psi(x^{-})}{2} = x-\sum_{\rho}\frac{x^{\rho}}{\rho} - \log2\pi -\frac{1}{2}\log\Bigl(1-\frac{1}{x^{2}}\Bigr),
\]
the sum being over the non-trivial zeros of the Riemann zeta function. Hence, at least as distributions one has $-\sum_{\rho}x^{\rho-1} = \sum_{n}\Lambda(n)\delta_{n}(x) + \frac{1}{x^{3}-x^{2}}-1$, so one expects positive divergence in $-\sum_{\rho}x^{\rho-1}$ when $x$ is a prime power and massive cancellation otherwise. However, such a behavior seems very hard to ``cook up'' for a general well-chosen sequence of zeros $(\rho_{k})_{k}$, without knowing in advance that they come from a positive prime measure!}.

\medskip

In the construction of Diamond, Montgomery, and Vorhauer \cite{DMV}, these issues are dealt with in a very ingenious manner. Instead of the simple $\frac{s-\rho}{s}$, another function is used to introduce a zero at $s=\rho$, namely
\begin{equation}
\label{G}
	G(z) \coloneqq 1- \frac{\e^{-z} - \e^{-2z}}{z}.
\end{equation}
This is an entire function with zeros at $z=0$ and $z = z_{\pm j} = x_{j} \pm \I y_{j}$, $j\in \N_{>0}$, where (see \cite[Lemmas 1--2]{DMV})
\begin{equation}
\label{zeros G}
	x_{j} < -\frac{\log(\pi j/2)}{2}, \quad x_{j} = -\frac{\log(\pi j)}{2} + O(j^{-1/2}), \quad \pi j < y_{j} < \pi(j+1).
\end{equation}
Hence for $a>0$, $G(a(s-\rho))$ has a zero at $s=\rho$ (and infinitely many to the left of $\rho$).

In \cite{DMV}, it is shown that $\log G(z)$ can be expressed as a Mellin transform:
\begin{equation}
\label{Mellin G}
	\log G(z) = -\int_{1}^{\infty}g(x)x^{-z-1}\dif x,
\end{equation}
where $g$ is a function supported on $[\e, \infty)$. As will become evident later on, the fact that $g$ is supported away from $x=1$ is crucial.
Below we collect the other necessary properties of $g$ required for our analysis. They are proved in \cite[Section 4]{DMV}; see also \cite[Section 17.6]{DZ16}.
\begin{lemma}
\label{function g}
The function $g$ is nonnegative, continuous on $(\e^{2}, \infty)$ and continuously differentiable on $(\e^{4}, \infty)$. On the interval $(\e^{m}, \e^{m+1})$ ($m \in \N_{>0}$), $g(x)$ is a polynomial in $\log x$ of degree at most $m-1$. Furthermore, $g(x)$ satisfies the estimates
\begin{align*}
	g(x)\log x 	&= 1 + O(x^{-1/2}), \quad &&x\ge \e^{2}\\
	\bigl(g(x)\log x\bigr)'	&\ll x^{-3/2}, \quad &&x\ge \e^{5}.
\end{align*}
\end{lemma}

Let us now consider arbitrary $\rho\in\C$ and $a>0$. Multiplying a zeta function with the factor $G(a(s-\rho))$ introduces a zero at $s=\rho$, and corresponds with adding to $\Pi(x)$ the term
\[
	-\int_{1}^{x}\frac{g(u^{1/a})}{a}u^{\rho-1}\dif u.
\]
Note that this is nonzero only if $x>\e^{a}$.
\begin{lemma}
\label{estimate Ik}
Assume that $\rho = \beta+\I\gamma$ with $\beta \in [1/2, 1)$ and $\gamma\ge1$ and let $a\ge4$. Then for $x\ge \e^{5a}$ we have
\[
	-\int_{1}^{x}\frac{g(u^{1/a})}{a}(\log u)u^{\beta-1}\cos(\gamma\log u)\dif u = -\frac{x^{\beta}\sin(\gamma\log x)}{\gamma} 
	+ O\biggl(\frac{x^{\beta-1/(2a)}}{\gamma}+\frac{x^{\beta}}{\gamma^{2}}+\e^{5a\beta}\biggr),
\]
while for $\e^{a} \le x < \e^{5a}$, 
\[
	-\int_{1}^{x}\frac{g(u^{1/a})}{a}(\log u)u^{\beta-1}\cos(\gamma\log u)\dif u = O\biggl(\frac{x^{\beta}}{\gamma}\biggr).
\] 
Here, the implicit $O$-constants are absolute; they do not depend on $\rho$ or $a$.
\end{lemma}
\begin{proof}
If $x\ge \e^{5a}$, we split the integral into the part from $1$ tot $\e^{5a}$ and the part from $\e^{5a}$ to $x$. The first integral is $O(\e^{5a\beta})$, while the second can be estimated via integrating by parts and using the estimates from Lemma \ref{function g}.

If $x < \e^{5a}$, we split the integral in integrals over the intervals $[\e^{ma}, \min(\e^{(m+1)a}, x)]$, $1\le m \le 4$. On each such interval, we write $g(u^{1/a})\log u^{1/a} = P_{m}(\log u^{1/a})$, where $P_{m}$ is a polynomial of degree at most $m$ (cfr.\ Lemma \ref{function g}), and integrate by parts.
\end{proof}
If we now consider an infinite sequence of zeros $(\rho_{k})_{k}$ and a corresponding sequence $(a_{k})_{k}$ increasing to infinity, then we have no issues of convergence in the contributions to $\psi_{\MB}(x)$ and $\psi_{\MB}'(x)$: in the series
\[
	-\sum_{k}\int_{1}^{x}\frac{g(u^{1/a_{k}})}{a_{k}}(\log u)u^{\rho_{k}-1}\dif u \quad\text{and}\quad -\sum_{k}\frac{g(x^{1/a_{k}})}{a_{k}}(\log x)x^{\rho_{k}-1},
\]
only the terms with $\e^{a_{k}} \le x$ are nonzero, owing to the support of $g$. So, for each fixed $x$, only finitely many terms contribute to the series.


\section{Proof of Theorem \ref{sharpness theorem}}
\label{Proof of sharpness theorem}
\subsection{The setup}
We recall that we assume that $f$ is an (eventually) $C^{1}$, decreasing, and strictly convex function satisfying $f(u) \succ 1/u$, which is either regularly varying of index $-\alpha$ ($0<\alpha\le1$) or slowly varying, in which case it also satisfies \eqref{decrease |f'|}. We introduce the following parameters:
\begin{equation}
	\gamma_{k} \coloneqq \exp(4^{k}), \quad \ell_{k} \coloneqq \frac{1}{f(\log\gamma_{k})}, \quad \beta_{k} \coloneqq 1-\frac{1}{\ell_{k}}, \quad a_{k} \coloneqq \Bigl(1-\frac{1}{\ell_{k}}\Bigr)\log\gamma_{k}.\label{parameters}
\end{equation}
We will define a zeta function having a bunch of zeros near $\beta_{k}+\I\gamma_{k}$, for every integer $k\ge k_{0}$ with $k_{0}$ sufficiently large. We set
\begin{equation}
\label{zeros}
	M_{k} \coloneqq \biggl\lfloor \exp\biggl(\frac{a_{k}}{\ell_{k}}-\sqrt{\frac{a_{k}}{\ell_{k}}}\biggr) \biggr\rfloor, \quad \gamma_{k,m}\coloneqq \gamma_{k} + c_{k,m}\frac{m\gamma_{k}}{M_{k}}, \quad \rho_{k,m} \coloneqq \beta_{k} + \I \gamma_{k,m},
\end{equation}
for  $m=0, 1, \dotsc, M_{k}$. Here, $c_{k,m}$ are real numbers with $c_{k,m} = 1+o(1)$ which will be determined later. From the outset we note that $a_{k}/\ell_{k} \sim \log\gamma_{k}f(\log\gamma_{k}) \to \infty$ and $M_{k} \le \gamma_{k}^{1/\ell_{k}}$.

Finally we define
\begin{align}
	\zeta_{\mc}(s) 	&\coloneqq \frac{s}{s-1}\prod_{k=k_{0}}^{\infty}\prod_{m=0}^{M_{k}}G\bigl(a_{k}(s-\rho_{k,m})\bigr)G\bigl(a_{k}(s-\overline{\rho_{k,m}})\bigr); \label{product}\\
	\Pi_{\mc}(x) 	&\coloneqq \Li(x) -2\sum_{k=k_{0}}^{\infty}\sum_{m=0}^{M_{k}}\int_{1}^{x}\frac{g(u^{1/a_{k}})}{a_{k}}u^{-1/\ell_{k}}\cos(\gamma_{k,m}\log u)\dif u. \label{Pic}
\end{align}
Note that in view of \eqref{Mellin G}, 
\[
	\zeta_{\mc}(s) = \exp\int_{1}^{\infty}x^{-s}\dif\Pi_{\mc}(x).
\]

\subsection{The template Beurling number system}\label{template system}
First of all we show that $\Pi_{\mc}(x)$ is non-decreasing. For this it suffices to show that $\psi_{\mc}'(x) \ge 0$, where $\psi_{\mc}(x)\coloneqq \int_{1}^{x}\log u\dif\Pi_{\mc}(u)$. Let $x\ge\e^{a_{k_{0}}}$, and let $K$ be such that $\e^{a_{K}} \le x < \e^{a_{K+1}}$. Then 
\[
	\psi_{\mc}'(x) = 1-\frac{1}{x} - 2\sum_{k=k_{0}}^{K}\sum_{m=0}^{M_{k}}\bigl(g(x^{1/a_{k}})\log x^{1/a_{k}}\bigr)\cos(\gamma_{k,m}\log x) x^{-1/\ell_{k}},
\]
so that from $g(u)\log u \ll 1$ (see Lemma \ref{function g}),
\begin{align*}
	\psi_{\mc}'(x) 	&\ge 1-\frac{1}{x} +O\Biggl(\sum_{k=k_{0}}^{K}M_{k}x^{-1/\ell_{k}}\Biggr) \\
				&= 1 - \frac{1}{x} +
				 O\Biggl(\sum_{k=k_{0}}^{K}\exp\biggl\{\biggl(1-\frac{1}{\sqrt{a_{k}/\ell_{k}}}\biggr)\frac{a_{k}}{\ell_{k}} - \frac{a_{K}}{\ell_{k}}\biggr\}\Biggr).
\end{align*}
The term corresponding with $k=K$ is $\exp(-\sqrt{a_{K}/\ell_{K}}) = o(1)$. For the terms with $k_{0}\le k < K$, we note that $a_{k}\sim 4^{k}$, so that
\begin{align*}
	\sum_{k=k_{0}}^{K}\exp\biggl\{\biggl(1-\frac{1}{\sqrt{a_{k}/\ell_{k}}}\biggr)\frac{a_{k}}{\ell_{k}} - \frac{a_{K}}{\ell_{k}}\biggr\}
	&\le o(1) + \sum_{k=k_{0}}^{K-1}\exp\Bigl(-\frac{4^{K}}{2\ell_{k}}\Bigr) \\
	&\le o(1) + \sum_{k=k_{0}}^{K-1}\exp\Bigl(-\frac{4^{K-k}}{2}\cdot 4^{k}f(4^{k})\Bigr) = o(1),
\end{align*}
as $4^{k}f(4^{k}) \to \infty$.
Hence, $\psi_{\mc}'(x) \ge 0$ for all $x\ge1$ if $k_{0}$ is sufficiently large.
\medskip

If we now set $\dif N_{\mc}(x) \coloneqq \exp^{\ast}(\dif\Pi_{\mc}(x))$, $N_{\mc}(x) = \int_{1^{-}}^{x}\dif N_{\mc}$, we have that $(\Pi_{\mc}(x), N_{\mc}(x))$ is a Beurling number system in the extended sense with associated Beurling zeta function $ \zeta_{\mc}(s)$.

\subsection{Bounding $\zeta_{\mc}(s)$}\label{bounding zetac}
Next we consider the infinite product for $\zeta_{\mc}$. Let $s=\sigma+\I t$ with $\sigma \ge1/2$ and $t\ge0$. Recalling the definition of $G$ \eqref{G} we have
\begin{align}
	G\bigl(a_{k}(s-\rho_{k,m})\bigr) &\ll 1, \quad &&\text{if } a_{k}\abs{s-\rho_{k,m}} \le 1;\\
	G\bigl(a_{k}(s-\rho_{k,m})\bigr) &=1 + O\biggl(\frac{1+\exp\bigl(2a_{k}(1-\sigma-\frac{1}{\ell_{k}})\bigr)}{a_{k}\abs{s-\rho_{k,m}}}\biggr); &&
\end{align}
Let $t>0$ be fixed. By the rapid increase of $\gamma_{k}$ there is at most one $k\ge k_{0}$ with $\gamma_{k}/2 \le t \le 3\gamma_{k}$. For all other values of $k$ one has $\abs{s-\rho_{k,m}} \gg \gamma_{k}$.

Suppose now $\gamma_{k'}/2 \le t \le 3\gamma_{k'}$ for some (unique) $k'\ge k_{0}$. There is at most one value of $m$, $0\le m\le M_{k'}$ with $\gamma_{k'} + (m-1/2)\gamma_{k'}/M_{k'} \le t < \gamma_{k'} + (m+1/2)\gamma_{k'}/M_{k'}$, which we denote by $m'$. If there is no such $m$, we set $m'=-1$ or $m'=M_{k'}+1$ depending on whether $t<\gamma_{k'}-\gamma_{k'}/(2M_{k'})$ or $t\ge \gamma_{k'} + (M_{k'}+1/2)\gamma_{k'}/M_{k'}$. Recalling \eqref{zeros} with $c_{k',m}=1+o(1)$, we have for $m\neq m'$ that 
$\abs{\gamma_{m,k'} - t} \gg \abs{m-m'}\gamma_{k'}/M_{k'}$. Since $\e^{a_{k}}M_{k}/\gamma_{k} \le 1$ we obtain
\[
	\sum_{m\neq m'}^{M_{k'}}\frac{\e^{a_{k'}}M_{k'}}{a_{k'}\abs{m-m'}\gamma_{k'}} + \sum_{k\neq k'}\frac{\e^{a_{k}}M_{k}}{a_{k}\gamma_{k}} 
	\ll \frac{1}{\ell_{k'}} + \sum_{k=k_{0}}^{\infty}4^{-k}  \ll 1.
\]
Hence the product \eqref{product} converges locally uniformly on $\Re s \ge 1/2$, and we obtain the bounds
\begin{align}
	\zeta_{\mc}(s) 	&\ll 1 + \frac{1}{|s-1|}, \quad &&\text{if } \forall k, m: \abs{t-\gamma_{k, m}} > \frac{\gamma_{k}}{2M_{k}}; \label{bound zetac 1}\\
	\zeta_{\mc}(s) 	&\ll \Bigl(1+t^{2(1-1/\ell_{k})(1-\sigma)}\Bigr)\min\biggl\{1, \frac{1}{\abs{t-\gamma_{k,m}}\log t}\biggr\}, \quad &&\text{if } \abs{t-\gamma_{k, m}} \le \frac{\gamma_{k}}{2M_{k}}. \label{bound zetac 2}
\end{align}
In particular, if $T\ge 1$,
\begin{equation}
\label{L1 bound zetac}
	\int_{T}^{2T}\abs{\zeta_{\mc}(\sigma+\I t)}\dif t \ll T, \quad \text{uniformly for } \sigma \ge 1/2.
\end{equation}

\subsection{The location of the zeros}\label{location zeros}
We now consider the zeros of $\zeta_{\mc}$. By construction, the zeros $\rho_{k,0}$ and $\overline{\rho_{k,0}}$ lie on the contour $\sigma=1-\eta(\abs{t})$. As $\eta$ is decreasing, the zeros $\rho_{k,m}$ and $\overline{\rho_{k,m}}$ with $m>0$ lie to the left of this contour. We claim that all others lie also to the left of this contour. In view of \eqref{zeros G}, \eqref{parameters}, \eqref{zeros}, and \eqref{product}, every such zero in the region $\{ s: \Re s \ge 1/2, \Im s \ge 0\}$ is of the form
\[
	\rho_{k, m, \pm j} \coloneqq 1-\frac{1}{\ell_{k}} + \frac{x_{j}}{a_{k}} + \I\Bigl(\gamma_{k,m}\pm \frac{y_{j}}{a_{k}}\Bigr), \quad k\ge k_{0}, \quad 0\le m\le M_{k}, \quad 1\le j \le J_{k},
\]
with $J_{k} \sim (1/\pi)\gamma_{k}^{(1-2/\ell_{k})(1-1/\ell_{k})}$.
As $\eta$ is decreasing, it suffices to check that $\Re \rho_{k, 0,-j} < 1 - f(\log (\Im \rho_{k, 0, -j}))$. For $j \le J_{k}$ we have
\[
	\log\Bigl(\gamma_{k} - \frac{y_{j}}{a_{k}}\Bigr) > \log\Bigl(\gamma_{k} - \frac{\pi j}{a_{k}}\Bigr) \ge \lambda \log \gamma_{k}	
\]
say, where $\lambda = 1- \frac{1}{a_{k}\log\gamma_{k}}$.
From the assumption $f(\lambda u) \sim \lambda^{-\alpha}f(u)$ with $0<\alpha \le 1$ or $f(\lambda u) \sim f(u)$, and the fact that this relation holds uniformly for $\lambda$ in compact subsets of $(0,\infty)$, we see that, if $k_{0}$ is sufficiently large in terms of $f$ and $k\ge k_{0}$,
\[
	f(\log(\Im\rho_{k,-j})) < f(\lambda \log\gamma_{k}) < \lambda^{-2}f(\log \gamma_{k}) \le\Bigl(1+\frac{4}{a_{k}\log\gamma_{k}}\Bigr)\frac{1}{\ell_{k}}.
\]
It now indeed follows that 
\[
	\Re \rho_{k, 0, -j} \le 1- \frac{1}{\ell_{k}} -\frac{\log(\pi j/2)}{2a_{k}} \le 1- \frac{1}{\ell_{k}} - \frac{4}{a_{k}\ell_{k}\log\gamma_{k}} < 1-f(\log (\Im\rho_{k, 0, -j})).
\]

Finally we note the following concerning the number of zeros of $\zeta_{\mc}$. For $1/2 < \sigma < 1- 1/\ell_{k}$, the number of zeros in the rectangle $[\sigma, 1] \times \I[\gamma_{k}/2, 3\gamma_{k}]$ is 
\[
	\sim \frac{M_{k}}{\pi}\gamma_{k}^{2(1-\sigma-1/\ell_{k})(1-1/\ell_{k})}
\]
for $k$ sufficiently large. In particular, for every fixed $\delta > 0$,
\begin{equation}
\label{number of zeros}
	N(\sigma, T) = \Omega\bigl(T^{2(1-\sigma) - \delta}\bigr), \quad \text{uniformly for } \sigma \in [1/2, 1-\delta/2].
\end{equation}

\subsection{The asymptotics of $\psi_{\mc}$}
\label{asymptotics psic}
We now analyze the asymptotic behavior of $\psi_{\mc}(x) = \int_{1}^{x}\log u \dif \Pi_{\mc}(u)$. Suppose $x$ is large and fixed, and let $K\ge k_{0}$ be such that $\e^{a_{K}} \le x < \e^{a_{K+1}}$. From \eqref{Pic} and the fact that $\supp g \subseteq [\e, \infty)$, we get that
\[
	\psi_{\mc}(x) = x - 1 - \log x - 2\sum_{k=k_{0}}^{K}\sum_{m=0}^{M_{k}}\int_{\e^{a_{k}}}^{x}\bigl(g(u^{1/a_{k}})\log u^{1/a_{k}}\bigr)u^{\beta_{k}-1}\cos(\gamma_{k,m}\log u)\dif u.
\]
If $k\le K-2$ then $\e^{5a_{k}} \le \e^{5a_{K}/16} \le x^{5/16}$ and $M_{k} \le \e^{a_{k}/\ell_{k}} \le x^{1/16}$, so that by Lemma \ref{estimate Ik}
\begin{multline}
	\psi_{\mc}(x) = x - 1 - \log x - 2\sum_{k=k_{0}}^{K-2}\sum_{m=0}^{M_{k}}\Biggl(\frac{x^{\beta_{k}}\sin(\gamma_{k,m}\log x)}{\gamma_{k,m}} 
	+ O\biggl\{\frac{x^{\beta_{k}}}{\gamma_{k}}\biggl(x^{-1/(2a_{k})}+\frac{1}{\gamma_{k}}\biggr)\biggr\}\Biggr) \\
	+ O\biggl(Kx^{6/16}+M_{K-1}\frac{x^{\beta_{K-1}}}{\gamma_{K-1}} + M_{K}\frac{x^{\beta_{K}}}{\gamma_{K}}\biggr).
	\label{psic sum}
\end{multline}
At first, note that $K\ll \log_{2} x$, $\log \gamma_{K-1} = (1/16)\log\gamma_{K+1} \ge (1/17)\log x$ say, $\log \gamma_{K} \ge (1/5)\log x$, and $M_{K} \le \e^{a_{K}/\ell_{K}} \le x^{1/\ell_{K}}$. Hence the final error term above is $O(x^{16/17})$.

Next, recalling \eqref{definition h} and \eqref{parameters}, 
\[
	\frac{x^{\beta_{k}}}{\gamma_{k}} = x\exp\Bigl(-\frac{1}{\ell_{k}}\log x -\log\gamma_{k}\Bigr) = x\exp\bigl(-h(\log\gamma_{k}, x)\bigr).
\] 
The function $h(u, x)$ attains its minimum $\omega(x)$ when $u=u_{0}(x)$. Let  
\begin{equation}
\label{definition kx}
	\kappa_{x} = \frac{\log u_{0}(x)}{\log 4}, \quad k_{x} = 
		\begin{cases}
			\lfloor \kappa_{x}\rfloor		&\text{if } \kappa_{x} - \lfloor \kappa \rfloor \le 1/2, \\
			\lceil \kappa_{x} \rceil		&\text{if } \kappa_{x} - \lfloor \kappa \rfloor > 1/2.
		\end{cases}
\end{equation}
In view of \eqref{asymp u-} or \eqref{u- o} we obtain $u_{0}(x) \ll f(u_{0}(x))\log x = o(\log x)$, so that in particular, $k_{x} < K-2$ if $x$ is sufficiently large. Note that $\log\gamma_{\kappa_{x}} = u_{0}(x)$ (with the obvious definition of $\kappa_{x}$ is not an integer) and 
\begin{align*}
	k > k_{x} &\implies \log\gamma_{k} \ge 2u_{0}(x), \\
	k < k_{x} &\implies \log\gamma_{k} \le \frac{u_{0}(x)}{2}.
\end{align*}

With this information, we can estimate the terms in \eqref{psic sum} with $k\neq k_{x}$: from Lemma \ref{technical} (with $\lambda=2$) it follows that there exists some $\mu>0$ so that
\[
	M_{k}\frac{x^{\beta_{k}}}{\gamma_{k}} \le x\exp\bigl(-h(\log\gamma_{k}, \log x) + f(\log\gamma_{k})\log \gamma_{k}\bigr) \ll x\exp\bigl(-\omega(x) - \mu u_{0}(x)\bigr), \quad k\neq k_{x}.
\]
Next, as $a_{k_{x}} \asymp \log \gamma_{k_{x}} \asymp u_{0}(x) = o(\log x)$, $x^{-1/(2a_{k_{x}})} \to 0$. Hence it follows that
\begin{equation}
\label{psic intermediate}
	\psi_{\mc}(x) = x -2\sum_{m=0}^{M_{k_{x}}}\Biggl(\frac{x^{\beta_{k_{x}}}\sin(\gamma_{k_{x}, m}\log x)}{\gamma_{k_{x}, m}}+ o\Bigl(\frac{x^{\beta_{k_{x}}}}{\gamma_{k_{x}}}\Bigr)\Biggr) + o\bigl(x\e^{-\omega(x)}\bigr).
\end{equation}
We now show that the above som has the desired order from Theorem \ref{sharpness theorem}, for suitably chosen arbitrarily large values of $x$.

Let $k\ge k_{0}$ be large and fixed and set
\begin{equation}
\label{xk}
	x_{k} \coloneqq u_{0}^{-1}(\log \gamma_{k}) = u_{0}^{-1}(4^{k}),
\end{equation}
where $u_{0}^{-1}$ denotes the inverse function of the increasing function $u_{0}$. With this definition we obtain $u_{0}(x_{k}) = \log\gamma_{k}$ and $k_{x_{k}} = \kappa_{x_{k}} = k$. Next we set $\tilde{x}_{k} = \e^{-\epsilon_{k}}x_{k}$, where $0 \le \epsilon_{k} < 2\pi/\gamma_{k}$ is chosen so that $\sin(\gamma_{k}\log \tilde{x}_{k}) = (-1)^{k}$. Finally, recalling $\gamma_{k,m} = \gamma_{k} + c_{k,m}m\gamma_{k}/M_{k}$, we choose $c_{k, m} = 1+o(1)$ so that $\gamma_{k, m}\log \tilde{x}_{k}-\gamma_{k}\log \tilde{x}_{k} \in 2\pi\Z$. With these choices we obtain from \eqref{psic intermediate} that
\[
	\psi_{\mc}(\tilde{x}_{k}) = \tilde{x}_{k} + 2(-1)^{k}\sum_{m=0}^{M_{k}}\frac{(\tilde{x}_{k})^{\beta_{k}}}{\gamma_{k,m}}(1+o(1)) + o\bigl(\tilde{x}_{k}\e^{-\omega(\tilde{x}_{k})}\bigr).
\]
The error we make by replacing $x_{k}$ with $\tilde{x}_{k}$ is negligible: 
\begin{gather*}
	\frac{(\tilde{x}_{k})^{\beta_{k}}}{\gamma_{k,m}} \asymp \frac{(x_{k})^{\beta_{k}}}{\gamma_{k}} = x_{k}\exp\bigl(-\omega(x_{k})\bigr) \asymp \tilde{x}_{k}\exp\bigl(-\omega(x_{k})\bigr); \\
	\abs{\omega(\tilde{x}_{k}) - \omega(x_{k})} \le \abs*{\log\frac{\tilde{x}_{k}}{x_{k}}} \ll \frac{1}{\gamma_{k}}, \quad \text{see \eqref{omega slow variation}};\\
	f\bigl(u_{0}(\tilde{x}_{k})\bigr) \sim (1+\alpha)\frac{\omega(\tilde{x}_{k})}{\log\tilde{x}_{k}} \sim (1+\alpha)\frac{\omega(x_{k})}{\log x_{k}} \sim f\bigl(u_{0}(x_{k})\bigr);
\end{gather*}
where in the last equation we again used \eqref{asymp u-} if $f$ is regularly varying of index $-\alpha$ and \eqref{u- o} if $f$ is slowly varying (in which case we set $\alpha=0$ above).
Finally, for every $\delta>0$ it holds for sufficiently large $k$ that
\begin{align*}
	M_{k} 	&\ge \exp\Bigl((1-\delta)\frac{\log\gamma_{k}}{\ell_{k}}\Bigr) = \exp\bigl((1-\delta)u_{0}(x_{k})f\bigl(u_{0}(x_{k})\bigr)\bigr) \\
			&\ge \exp\bigl((1-\delta)u_{0}(\tilde{x}_{k})f\bigl(u_{0}(x_{k})\bigr)\bigr) \ge \exp\bigl((1-2\delta)u_{0}(\tilde{x}_{k})f\bigl(u_{0}(\tilde{x}_{k})\bigr)\bigr).
\end{align*}
Putting all this together we obtain for every $\delta>0$:
\begin{align}
\label{asymptotic psic +}
	&\limsup_{k\to \infty} \frac{\psi_{\mc}(\tilde{x}_{2k}) - \tilde{x}_{2k}}{\tilde{x}_{2k}\exp\bigl\{-\omega(\tilde{x}_{2k}) + (1-\delta)f\bigl(u_{0}(\tilde{x}_{2k})\bigr)u_{0}(\tilde{x}_{2k})\bigr\}} = \infty,\\
\label{asymptotic psic -}	
	&\liminf_{k\to \infty} \frac{\psi_{\mc}(\tilde{x}_{2k+1}) - \tilde{x}_{2k+1}}{\tilde{x}_{2k+1}\exp\bigl\{-\omega(\tilde{x}_{2k+1}) + (1-\delta)f\bigl(u_{0}(\tilde{x}_{2k+1})\bigr)u_{0}(\tilde{x}_{2k+1})\bigr\}} = -\infty.
\end{align}

\subsection{The discrete approximation and conclusion of the proof}\label{discrete approximation}
To conclude the proof, we will approximate the system in the extended sense $(\Pi_{\mc}, N_{\mc})$ by an actual discrete system, by which we mean a sequence of Beurling generalized primes $\MP = (p_{1}, p_{2}, \dotsc)$ for which $\psi_{\MP}(x)$, $N_{\MP}(x)$, and $\zeta_{\MP}(s)$ are sufficiently close to $\psi_{\mc}(x)$, $N_{\mc}(x)$, and $\zeta_{\mc}(s)$ respectively. For it, we employ the random approximation procedure from \cite{BV24}, summarized in the following
\begin{theorem}[{\cite[Theorem 1.2]{BV24}}]
\label{discretization procedure}
Let $F$ be a non-decreasing right-continuous function tending to $\infty$ with $F(1)=0$ and $F(x) \ll x/\log x$. Then there exists a sequence of generalized primes $\MP = (p_{j})_{j\ge1}$ such that $\abs{\pi_{\MP}(x) - F(x)} \le 2$ and such that for any real $t$ and any $x\ge1$
\begin{equation}
\label{probabilistic approximation}
	\abs[\Bigg]{\sum_{p_{j}\le x}p_{j}^{-\I t} - \int_{1}^{x}u^{-\I t}\dif F(u)} \ll \sqrt{x} + \sqrt{\frac{x\log(\abs{t}+1)}{\log(x+1)}}.
\end{equation}
\end{theorem}
We apply this theorem with $F(x) = \Pi_{\mc}(x)$. Hence we obtain a sequence of Beurling primes $\MP$ with $\pi_{\MP}(x) = \Pi_{\mc}(x) + O(1)$. From this we deduce
\begin{align}
	\Pi_{\MP}(x) 	&= \pi_{\MP}(x) + O(\sqrt{x}) = \Pi_{\mc}(x) + O(\sqrt{x}), \label{Pi-pi}\\
	\psi_{\MP}(x) 	&= \int_{1}^{x}\log u\dif\Pi_{\MP}(u) = \psi_{\mc}(x) + O(\sqrt{x}\log x)\nonumber,
\end{align}
so that the second assertion of Theorem \ref{sharpness theorem} follows immediately from \eqref{asymptotic psic +} and \eqref{asymptotic psic -}.

\medskip

Next we write
\begin{align*}
	\log\zeta_{\MP}(s) - \log\zeta_{\mc}(s) 	&= \int_{1}^{\infty}x^{-s}\bigl(\dif\pi_{\MP}(x)-\dif\Pi_{\mc}(x)\bigr) + \int_{1}^{\infty}x^{-s}\bigl(\dif\Pi_{\MP}(x)-\dif\pi_{\MP}(x)\bigr)\\
									&\eqqcolon Z_{1}(s) + Z_{2}(s).
\end{align*}
The estimate \eqref{probabilistic approximation} and a short calculation show that $Z_{1}(s)$ defines an analytic function in the region $\sigma = \Re s > 1/2$, where it satisfies the estimate
\[
	Z_{1}(\sigma+\I t) \ll \frac{1}{\sigma-1/2} + \sqrt{\frac{\log(\abs{t}+1)}{\sigma-1/2}}.
\]
Similarly, as $\dif\Pi_{\MP}(x)-\dif\pi_{\MP}(x)$ is a positive measure and in view of \eqref{Pi-pi}, also $Z_{2}(s)$ is analytic for $\sigma>1/2$ and satisfies $Z_{2}(\sigma+\I t)\ll (\sigma-1/2)^{-1}$. Combining both estimates yields
\begin{equation}
\label{comparison zetaP zetac}
	\zeta_{\MP}(s) = \zeta_{\mc}(s)\e^{Z_{1}(s)+Z_{2}(s)} = \zeta_{\mc}(s)\exp\biggl(O\biggl(\frac{1}{\sigma-1/2} + \sqrt{\frac{\log(\abs{t}+1)}{\sigma-1/2}}\biggr)\biggr), \quad \sigma>1/2.
\end{equation}
We conclude that, apart from the pole at $s=1$, $\zeta_{\MP}(s)$ has analytic continuation to the half-plane $\sigma>1/2$, where it has the same zeros as $\zeta_{\mc}(s)$, so that also the third assertion of Theorem \ref{sharpness theorem} is proved.

\medskip

It remains to show that $N_{\MP}(x) = Ax + O_{\eps}(x^{1/2+\eps})$, for some $A>0$ and every $\eps>0$. To show this we will use Perron inversion as in \cite[Section 17.8]{DZ16}:
\[
	N_{\MP}(x) \le \int_{x}^{x+1}N_{\MP}(u)\dif u = \frac{1}{2\pi\I}\int_{2-\I\infty}^{2+\I\infty}\zeta_{\MP}(s)\frac{(x+1)^{s+1}-x^{s+1}}{s(s+1)}\dif s.
\]
We shift the contour of integration to $\Re s = \sigma_{0}\coloneqq 1/2 + (\log x)^{-1/3}$. Letting $A\coloneqq \res_{s=1}\zeta_{\MP}(s)$,  we obtain
\[
	N_{\MP}(x) \le Ax + \frac{A}{2} + \frac{1}{2\pi\I}\int_{\sigma_{0}-\I\infty}^{\sigma_{0}+\I\infty}\zeta_{\mc}(s)\e^{Z_{1}(s)+Z_{2}(s)}\frac{(x+1)^{s+1}-x^{s+1}}{s(s+1)}\dif s.
\]
We split the integral in dyadic intervals. Employing \eqref{L1 bound zetac} and the bounds
\begin{gather*}
	\abs[\big]{\e^{Z_{1}(s)+Z_{2}(s)}} \le \exp\Bigl\{c\bigl((\log x)^{1/3} + (\log T)(\log x)^{1/6}\bigr)\Bigr\}, \quad s = \sigma_{0} + \I t, \quad T\le \abs{t} \le 2T \\
	\abs[\big]{(x+1)^{s+1}-x^{s+1}} \ll \begin{cases}
			Tx^{1/2}\exp((\log x)^{2/3})	&\text{if } T\le x,\\
			x^{3/2}\exp((\log x)^{2/3})		&\text{if } T > x, 
			\end{cases}		
\end{gather*}
a short calculation yields
\[
	N_{\MP}(x) \le Ax + O\Bigl(x^{1/2}\exp\bigl(c'(\log x)^{2/3}\bigr)\Bigr), \quad \text{for some } c'>0.
\]
The other inequality follows similarly upon writing $N_{\MP}(x) \ge \int_{x-1}^{x}N_{\MP}(u)\dif u$.
This concludes the proof of Theorem \ref{sharpness theorem}.

\subsection{Variant 1: slower growth of $\zeta$}

In the above example, the additional factor $\exp\bigl((1-\delta)f\bigl(u_{0}(x)\bigr)u_{0}(x)\bigr)$ was obtained by, roughly speaking, putting as much zeros as possible in the rectangle $[1/2, \beta_{k}]\times \I[\gamma_{k}, 2\gamma_{k}]$. The cost for doing this is the large growth of the resulting zeta function: in bounding $\zeta_{\mc}(s)$ pointwise in \eqref{bound zetac 2}, we could only barely improve upon the ``trivial'' convexity bound \eqref{convexity bound}. In particular, the analogue of the Lindel\"of hypothesis fails for these zeta functions.

If one is content with $\psi_{\MP}(x)$ oscillating of size $x\e^{-\omega(x)}$ only, then a simpler construction with fewer zeros suffices, resulting in a zeta function with much better pointwise bounds. In the next section we will employ this simpler construction to shed some light on the relationship between zero-free regions on the one hand, and pointwise bounds and the size of ``clusters'' of zeros of the zeta function on the other hand.

Instead of \eqref{parameters}, \eqref{product}, \eqref{Pic}, we now set
\begin{gather}
	\ell_{k} \coloneqq 4^{k}, \quad \beta_{k} \coloneqq 1-1/\ell_{k}, \quad \gamma_{k} \coloneqq \exp\bigl(f^{-1}(1/\ell_{k})\bigr), \quad \rho_{k} \coloneqq \beta_{k}+\I\gamma_{k}, \quad a_{k} \coloneqq \ell_{k}; \label{parameters1}\\
	\zeta_{1}(s) \coloneqq \frac{s}{s-1}\prod_{k=k_{0}}^{\infty}G\bigl(a_{k}(s-\rho_{k})\bigr)G\bigl(a_{k}(s-\overline{\rho_{k}})\bigr); \label{zeta1}\\
	\Pi_{1}(x) \coloneqq \Li(x) -2\sum_{k=k_{0}}^{\infty}\int_{1}^{x}\frac{g(u^{1/a_{k}})}{a_{k}}u^{-1/\ell_{k}}\cos(\gamma_{k}\log u)\dif u. \label{Pi1}
\end{gather}

The method of Subsection \ref{template system} no longer applies to see that $\Pi_{1}(x)$ is non-decreasing. However, the non-decreasingness can be proven in exactly the same way as \cite[Lemmas 17.20--17.21]{DZ16}, the only difference being the value of the $\gamma_{k}$. However, they only occur in the cosine, which is estimated trivially by $1$. Letting $\dif N_{1}(x) \coloneqq \exp^{\ast}(\dif\Pi_{1}(x))$, $N_{1}(x) = \int_{1^{-}}^{x}\dif N_{1}$, we have that $(\Pi_{1}(x), N_{1}(x))$ is a Beurling number system in the extended sense with associated zeta function
\[
	\exp\biggl(\int_{1}^{\infty}x^{-s}\dif\Pi_{1}(x)\biggr) = \zeta_{1}(s).
\]

Similarly as in Subsection \ref{bounding zetac} one has
\begin{align}
	\zeta_{1}(s) 	&\ll 1 + \frac{1}{|s-1|}, \quad &&\text{if } \forall k: \abs{t-\gamma_{k}} > \gamma_{k}/2; \nonumber\\
	\zeta_{1}(s) 	&\ll \biggl(1+\exp\biggl(\frac{2(1-\sigma)}{f(\log\gamma_{k})}\biggr)\biggr)\min\biggl\{1, \frac{f(\log\gamma_{k})}{\abs{s-\rho_{k}}}\biggr\}, &&\nonumber \\
				&\ll \biggl(1+\exp\biggl(\frac{2(1+o(1))(1-\sigma)}{f(\log t)}\biggr)\biggr)\min\biggl\{1, \frac{f(\log t)}{\abs{s-\rho_{k}}}\biggr\}, \quad &&\text{if } \abs{t-\gamma_{k}} \le \gamma_{k}/2. \label{bound zeta1}
\end{align}
In particular, as $1/f(u) = o(u)$, we have $\zeta_{1}(\sigma+ \I t) \ll_{\eps} \abs{t}^{\eps}$, for every $\eps>0$, so the analogue of the Lindel\"of hypothesis holds for $\zeta_{1}$.

The analysis of the zeros is similar to Subsection \ref{location zeros}. Regarding the oscillation of $\psi_{1}(x) \coloneqq \int_{1}^{x}\log u\dif\Pi_{1}(u)$, one can in this case actually show that
\begin{equation}
\label{asymptotic psi1}
	\limsup_{x\to\infty}\frac{\psi_{1}(x) -x}{x\exp\bigl(-\omega(x)\bigr)} = 2, \quad \liminf_{x\to\infty}\frac{\psi_{1}(x) -x}{x\exp\bigl(-\omega(x)\bigr)} = -2.
\end{equation}
This is done analogously as in Subsection \ref{asymptotics psic}, except that one sets
\[
	\kappa_{x} = \frac{\log\frac{1}{f(u_{0}(x))}}{\log 4}, \quad k_{x} = 
		\begin{cases}
			\lfloor \kappa_{x}\rfloor		&\text{if } \kappa - \lfloor \kappa_{x} \rfloor \le 1/2, \\
			\lceil \kappa_{x} \rceil		&\text{if } \kappa - \lfloor \kappa_{x} \rfloor > 1/2.
		\end{cases}
\]
Now one has $u_{0}(x) = f^{-1}(4^{-\kappa_{x}})$, and
\begin{align*}
	k > k_{x} &\implies \log\gamma_{k} \ge f^{-1}\biggl[\frac{1}{2}f\bigl(u_{0}(x)\bigr)\biggr] \ge (2+o(1))u_{0}(x), \\
	k < k_{x} &\implies \log\gamma_{k} \le f^{-1}\Bigl[2f\bigl(u_{0}(x)\bigr)\Bigr]  \le \biggl(\frac{1}{2}+o(1)\biggr)u_{0}(x);
\end{align*}
where the right-most equalities follow from the regular or slow variation of $f$ and the fact that $f^{-1}$ is decreasing. Hence one can similarly invoke Lemma \ref{technical} to deduce \eqref{asymptotic psi1}.

Finally, the fact that $N_{1}(x) = Ax + O_{\eps}(x^{1/2+\eps})$ for some $A>0$ and every $\eps>0$ can be shown similarly as in Subsection \ref{discrete approximation}, or alternatively from applying a Schnee--Landau-type theorem using the bound $\zeta_{1}(\sigma+\I t) \ll_{\eps}\abs{t}^{\eps}$.

\subsection{Variant 2: $o(x\e^{-\omega(x)})$ remainder}
In the first construction, the location of the zeros $\rho_{k,m} = \beta_{k} + \I \gamma_{k,m}$ was chosen in such a way that the different oscillations $x^{\rho_{k,m}}/\rho_{k,m}$ ($m=0, \dotsc, M_{k}$, $k$ fixed) in the explicit Riemann--von Mangoldt formula \eqref{explicit Riemann--von Mangoldt} exhibit constructive interference, producing the additional factor $\exp\bigl((1-\delta)f\bigl(u_{0}(x)\bigr)u_{0}(x)\bigr)$. It is also possible to introduce destructive interference, and actually obtain a non-trivial saving over $x\e^{-\omega(x)}$ in the error term of the PNT.

We will show this when $f(u)=u^{-\alpha}$ for some $\alpha \in (0,1)$. The parameters $\gamma_{k}, \ell_{k}, \beta_{k}, a_{k}$ remain the same as in \eqref{parameters}. Next we set $M_{k}=1$ and
\begin{equation}
\label{gammak1}
	\gamma_{k,0} = \gamma_{k}, \quad \gamma_{k,1} = \gamma_{k} + \frac{\pi}{\log x_{k}},  
\end{equation}
where $x_{k}$ is defined as in \eqref{xk}, so
\[
	x_{k} = u_{0}^{-1}(4^{k}) = \exp\biggl(\frac{4^{(1+\alpha)k}}{\alpha}\biggr),	
\]
as $u_{0}(x) = (\alpha\log x)^{\frac{1}{1+\alpha}}$ in this case. The Beurling number system is then defined via 
\begin{align*}
	\zeta_{2}(s) 	&\coloneqq \frac{s}{s-1}\prod_{k=k_{0}}^{\infty}\prod_{m=0}^{1}G\bigl(a_{k}(s-\rho_{k,m})\bigr)G\bigl(a_{k}(s-\overline{\rho_{k,m}})\bigr); \\
	\Pi_{2}(x) 		&\coloneqq \Li(x) -2\sum_{k=k_{0}}^{\infty}\sum_{m=0}^{1}\int_{1}^{x}\frac{g(u^{1/a_{k}})}{a_{k}}u^{-1/\ell_{k}}\cos(\gamma_{k,m}\log u)\dif u. 
\end{align*}

The analysis of the Subsections \ref{template system}, \ref{bounding zetac}, \ref{location zeros}, and \ref{discrete approximation} goes through without difficulties. The key difference with the first construction lies only in the analysis if $\psi_{2}(x)\coloneqq \int_{1}^{x}\log u\dif\Pi_{2}(u)$.

Let $x$ be large, and set again $k\coloneqq k_{x}$ as in \eqref{definition kx}. The analysis in Subsection \ref{asymptotics psic} can be followed to obtain
\begin{multline*}
	\psi_{2}(x) = x - 2x^{\beta_{k}}\biggl(\frac{\sin(\gamma_{k, 0}\log x)}{\gamma_{k, 0}} + \frac{\sin(\gamma_{k, 1}\log x)}{\gamma_{k, 1}}\biggr) \\
	+ O\biggl(\frac{x^{\beta_{k}-1/(2a_{k})}}{\gamma_{k}} + \frac{x^{\beta_{k}}}{(\gamma_{k})^{2}} + x\exp\bigl(-\omega(x) - \mu u_{0}(x)\bigr)\biggr),
\end{multline*}
for some $\mu>0$. Recall from \eqref{definition kx} that $k = k_{x} = \kappa_{x}+\vartheta_{x}$ with $\abs{\vartheta_{x}} \le 1/2$, so $\log\gamma_{k} = 4^{\vartheta_{x}}4^{\kappa_{x}} \asymp u_{0}(x) \asymp (\log x)^{\frac{1}{1+\alpha}}$. Also $\log x/a_{k} \asymp \log x / \log\gamma_{k} \asymp (\log x)^{\frac{\alpha}{1+\alpha}}$, and in any case $x^{\beta_{k}}/\gamma_{k} \ll x\e^{-\omega(x)}$. Hence we see that
\[
	\psi_{2}(x) = x - \frac{2x^{\beta_{k}}}{\gamma_{k}}\bigl(\sin(\gamma_{k, 0}\log x) + \sin(\gamma_{k, 1}\log x)\bigr) 
	+ O\Bigl\{x\exp\Bigl(-\omega(x) - c(\log x)^{\frac{\alpha}{1+\alpha}}\Bigr)\Bigr\},
\]
for some $c>0$. We will show that the second term above is $\ll x\e^{-\omega(x)}(\log x)^{-\frac{1}{2(1+\alpha)}}$. When $x$ is close to $x_{k} = u_{0}^{-1}(\log\gamma_{k})$, $x^{\beta_{k}}/\gamma_{k}$ is close to $x\e^{-\omega(x)}$, and we will obtain cancellation from the two sines, which are exactly out of phase for $x=x_{k}$, see \eqref{gammak1}. When $x$ is further removed from $x_{k}$, we cannot extract cancellation from the sum of sines anymore, and we will instead use that $x^{\beta_{k}}/\log \gamma_{k} = x\exp(-h(\log\gamma_{k}, x))$ becomes smaller compared to $x\e^{-\omega(x)}$, as $\log\gamma_{k} = u_{0}(x_{k})$ is then further removed from $u_{0}(x)$.

We start by writing $x = (x_{k})^{1+\nu}$. As $u_{0}(x) = 4^{\kappa_{x}} = 4^{k-\vartheta_{x}} = 4^{-\vartheta_{x}}u_{0}(x_{k})$, we have that 
\[	
	x = \exp\biggl(\frac{1}{\alpha}\bigl(u_{0}(x_{k})4^{-\vartheta_{x}}\bigr)^{1+\alpha}\biggr) = (x_{k})^{4^{-(1+\alpha)\vartheta_{x}}}, 
\]
so $4^{-\frac{1+\alpha}{2}} - 1 \le \nu \le 4^{\frac{1+\alpha}{2}} - 1$. We have
\begin{equation}
\label{u0x}
	u_{0}\bigl((x_{k})^{1+\nu}\bigr) = (1+\nu)^{\frac{1}{1+\alpha}}u_{0}(x_{k}) = (1+\nu)^{\frac{1}{1+\alpha}}\log \gamma_{k},
	 \quad f\bigl[u_{0}\bigl((x_{k})^{1+\nu}\bigr)\bigr] = (1+\nu)^{-\frac{\alpha}{1+\alpha}}f\bigl[u_{0}(x_{k})\bigr].
\end{equation}
If $\abs{\nu}$ is larger than some fixed $\delta>0$, then we can invoke Lemma \ref{technical} with $\lambda = (1+\delta)^{\frac{1}{1+\alpha}}$ or $\lambda = (1-\delta)^{-\frac{1}{1+\alpha}}$ to see that there exists a constant $\mu(\delta)>0$ for which 
\begin{equation*}
	\frac{x^{\beta_{k}}}{\gamma_{k}} = x\exp\bigl(-h(\log\gamma_{k}, x)\bigr) \le x\exp\bigl(-\omega(x) - \mu(\delta)u_{0}(x)\bigr).
\end{equation*}
Hence we may assume that $\abs{\nu} \le \delta$ for some small fixed $\delta>0$.

With \eqref{gammak1} and some trigonometry one finds
\[
	\sin(\gamma_{k,0}\log x)+\sin(\gamma_{k,1}\log x) = \bigl(1-\cos(\pi\nu)\bigr)\sin(\gamma_{k}\log x) -\sin(\pi\nu)\cos(\gamma_{k}\log x) \ll \abs{\nu}.
\]
On the other hand, for small $\nu$, in view of \eqref{u0x} and the fact that $f\bigl(u_{0}(x)\bigr)\log x = u_{0}(x)/\alpha$ we obtain
\begin{align*}
	\frac{x^{\beta_{k}}}{\gamma_{k}} 
		&= x\exp\biggl(-\omega(x) - \biggl(\frac{(1+\nu)^{-\frac{\alpha}{1+\alpha}}-1}{\alpha}+ (1+\nu)^{\frac{1}{1+\alpha}}-1\biggr)u_{0}(x)\biggr)\\
		&= x\exp\biggl(-\omega(x) - \biggl(\frac{\nu^{2}}{2(1+\alpha)}+O(\nu^{3})\biggr)u_{0}(x)\biggr)\\
		&\le x\exp\biggl(-\omega(x) - \frac{\nu^{2}u_{0}(x_{k})}{4(1+\alpha)}\biggr),
\end{align*}
provided that $\delta$ is sufficiently small, as $u_{0}(x) = \bigl(1+O(\abs{\nu})\bigr)u_{0}(x_{k})$ in view of \eqref{u0x}.
Combining both we find that 
\[
	\frac{\psi_{2}(x) - x}{x\e^{-\omega(x)}} \ll \exp\biggl(\log\abs{\nu}-\frac{u_{0}(x_{k})\nu^{2}}{4(1+\alpha)} \biggr).
\]
The function $\log\abs{\nu} - u_{0}(x_{k})\nu^{2}/(4(1+\alpha))$ is maximal when $\abs{\nu} = \sqrt{2(1+\alpha)/u_{0}(x_{k})}$, at which it takes the value $\log\sqrt{2(1+\alpha)/u_{0}(x_{k})}-1/2$, so that
\[
	\frac{\psi_{2}(x) - x}{x\e^{-\omega(x)}} \ll \exp\Biggl(\log\sqrt{\frac{2(1+\alpha)}{u_{0}(x_{k})}}\Biggr) \asymp (\log x_{k})^{-\frac{1}{2(1+\alpha)}} \asymp (\log x)^{-\frac{1}{2(1+\alpha)}}.
\]	
Summarizing we obtain
\begin{proposition}
\label{sharpness Revesz2}
For every $\alpha \in (0,1)$ there exists a Beurling number system $(\MP, \MN)$ satisfying
\begin{enumerate}
	\item $N_{\MP}(x) = Ax + O_{\eps}(x^{1/2+\eps})$ for some $A>0$ and every $\eps>0$;
	\item $\zeta_{\MP}(s)$ has infinitely many zeros on the contour $\sigma = 1- (\log\abs{t})^{-\alpha}$, and none to the right of this contour;
	\item the error term in the PNT satisfies
	\[ 
		\psi_{\MP}(x)-x \ll x\e^{-\omega(x)}(\log x)^{-\frac{1}{2(1+\alpha)}}. 
	\]	
\end{enumerate}
\end{proposition}
In particular, in Theorem \ref{Revesz2} $\eps$ cannot be taken to be zero.


\section{Some remarks on zero-distribution and growth of Beurling zeta functions}
\label{sec: zero-distribution and growth}
Finally we discuss some connections between zero-free regions on the one hand, and growth, zero-density estimates, and the size of zero clusters of the Beurling zeta function on the other hand.

\subsection{Landau's method for deducing zero-free regions}
Currently, the best technology available for producing zero-free regions is the century-old method of Landau \cite{La24}. Given an upper bound for $\abs{\zeta}$ near the $1$-line, one uses the Borel--Carath\'eodory lemma to upper bound $-\Re \frac{\zeta'}{\zeta}$ in terms of the $\abs{\zeta}$-bound and the real part of a hypothetical nearby zero. Juxtaposing this with the familiar ``3-4-1-inequality'' yields an upper bound for the real part of the hypothetical zero, i.e.\ a zero-free region.
\begin{theorem}[Hauptsatz of \cite{La24}]
\label{Hauptsatz}
Let $\vphi(t)$ and $\psi(t)$ be positive non-decreasing functions for $t \ge t_{0}$ satisfying $\vphi(t) \to \infty$, $\psi(t) \ge 1$, and 
\begin{equation}
\label{condition phi psi}
	\psi(t) \ll \e^{\vphi(t)}.
\end{equation}
Let $\dif\mathfrak{m}$ be a positive measure on $[1, \infty)$ for which $\int_{1^{-}}^{\infty}x^{-\sigma}\dif \mathfrak{m}(x)$ converges for $\sigma>1$. Suppose $F(s) = \exp\int_{1^{-}}^{\infty}x^{-s}\dif\mathfrak{m}(x)$ has analytic continuation to $\sigma \ge 1-1/\psi(\abs{t})$, with the exception of a simple pole at $s=1$, and suppose
\begin{equation}
\label{Landau bound}
	F(s) \ll \e^{\vphi(t)} \quad \text{for} \quad \sigma \ge 1-\frac{1}{\psi(t)}, \quad t\ge t_{0}.
\end{equation}
Then there is a constant $c>0$ so that
\begin{equation}
\label{Landau zero-free}
	F(s) \neq 0 \quad \text{for} \quad \sigma \ge 1-\frac{c}{\psi(2t+1)\vphi(2t+1)}, \quad t\ge t_{0}.
\end{equation}
\end{theorem}
\begin{remark}
\begin{itemize}
	\item Landau formulated his Hauptsatz for functions $F$ of the form $F(s) = \exp\sum_{n=1}^{\infty}a_{n}n^{-s}$ with $a_{n}\ge 1$, but the proof readily applies to the more general case. Also he considered the point $3t$ instead of $2t+1$. Typically $\vphi(t)$ and $\psi(t)$ are slowly varying, so the specific constant in front of $t$ is not so important.
	\item The standard works of Titchmarsh and Ivi\'c also contain this result for the Riemann zeta function specifically (see \cite[Theorem 3.10]{Ti86} and \cite[Lemma 6.8]{Iv03} respectively), but interestingly, they both require that $\vphi(t)\psi(t) = o(\e^{\vphi(t)})$, instead of Landau's slightly more general condition \eqref{condition phi psi} which is sufficient.
	\item The proof hinges on a clever use of a nonnegative trigonometric polynomial such as $3+4\cos t+\cos(2t)$. Landau investigated the use of other polynomials. Setting
	\[
		V \coloneqq \inf\biggl\{\frac{a_{1} + \dotsb + a_{n}}{(\sqrt{a_{1}}-\sqrt{a_{0}})^{2}} : a_{k}\ge 0, \quad a_{1}>a_{0}, \quad a_{0} + a_{1}\cos t + \dotsb + a_{n}\cos(nt) \ge 0\biggr\},
	\]
	he showed\footnote{This work predates his discovery using the ``local method'' \cite{La24} and employs the Hadamard factorization of $\zeta$ to obtain
	\[
		-\Re \frac{\zeta'(s)}{\zeta(s)} \le \frac{1}{2}\log t - \Re \sum_{\rho}\biggl(\frac{1}{\rho} + \frac{1}{s-\rho}\biggr) + O(1),
	\] yielding the better constant $2/V$ instead of $1/(8V)$ in the de la Vall\'ee Poussin zero-free region.} \cite[\S 79]{La09} for the Riemann zeta function that it has no zeros for $\sigma > 1 - \frac{2/V-\eps}{\log t}$, $t\ge t_{0}(\eps)$, for every $\eps>0$. Determining the value of $V$ is one of Landau's extremal problems, a historical survey of which may be found in \cite{Re07}. It is known \cite[Theorem 2]{MT15} that $34.468305 < V < 34.4889920009$.
	
	 Tracking the best constants in Landau's proof of his Hauptsatz, one may see that any $c < 1/(8V)$ in \eqref{Landau zero-free} is admissible.
\end{itemize}
\end{remark}
In this context it is interesting to consider the zeta function $\zeta_{1}(s)$. Setting $\frac{1}{\psi(t)} = a(t)f(\log t)$ for some function $a(t)$, we see from \eqref{bound zeta1} that with $\vphi(t) = (2+\delta)a(t)$,
\[
	\zeta_{1}(\sigma+\I t) \ll \e^{\vphi(t)}, \quad \text{for} \quad \sigma \ge 1 - \frac{1}{\psi(t)}, \quad t\ge t_{0}.
\]
Landau's theorem then yields that $\zeta_{1}(s)$ has no zeros for $\sigma \ge 1-(c/(2+\delta)+o(1))f(\log t)$, provided that $\psi(t) \ll \e^{\vphi(t)}$ and the monotonicity assumptions for $\psi$ and $\vphi$ are fulfilled. This is the case when setting for example $a(t) = \log\frac{1}{f(\log t)}$. Of course, by construction the ``true'' zero-free region of $\zeta_{1}(s)$ is $\sigma = 1 - f(\log t)$. Hence, at least in the class of Beurling zeta functions of Beurling number systems \emph{in the extended sense}, Landau's Theorem \ref{Hauptsatz} is sharp, up to the value of the constant $c$ in \eqref{Landau zero-free}. More precisely, we can summarize the above discussion as follows.
\begin{proposition}
\label{sharpness Landau}
Let $f$ be an eventually $C^{1}$, decreasing, and strictly convex function satisfying $1/u = o(f(u))$ as $u\to\infty$. Suppose further that $f$ is regularly varying of index $-\alpha$, $0<\alpha\le 1$, or that it is slowly varying and satisfies \eqref{decrease |f'|}. Then there exists a Beurling zeta function $\zeta_{\MB}(s)$ of a Beurling number system in the extended sense $\MB$ which has infinitely many zeros in the region $\sigma \ge 1- f(\log t)$, and for which there exist non-decreasing functions $\vphi(t) \to \infty$, $\psi(t) \ge 1$ satisfying \eqref{condition phi psi} and \eqref{Landau bound} (with $F(s) = \zeta_{\MB}(s)$), and for which
\[
	\frac{1}{\psi(2t+1)\vphi(2t+1)} \ge \frac{f(\log t)}{2+o(1)}.
\] 
\end{proposition}
Setting
\[
	\mathfrak{c} \coloneqq \sup\{c>0: c \text{ is admissible in \eqref{Landau zero-free}}\},
\]
we have $1/(8V) \le \mathfrak{c} \le 2$. Here $1/(8V) \approx 0.0036$. Based on the (loosely formulated) believe that the Diamond--Montgomery--Vorhauer-style examples are ``sharp up to a factor $1/2$'' (e.g.\ $\zeta_{1}(s)$ has analytic continuation to $\Re s >1/2$, not to $\Re s>0$), one may conjecture that $\mathfrak{c} = 1$.

\medskip

Finally we mention that we cannot show sharpness of Theorem \ref{Hauptsatz} in the smaller class of Beurling zeta functions of \emph{discrete} Beurling number systems, as the discretization procedure \ref{discretization procedure} introduces an extra factor of size $\exp\bigl(O(\sqrt{\log t})\bigr)$, which makes the above argument break down in some cases: to be able to use the same function $\vphi$, we would require $a(t) \gg \sqrt{\log t}$, but in this case $\psi(t) \ll (\log t)^{\alpha-1/2+o(1)}$, so that $\psi(t)$ cannot be a non-decreasing function $\ge 1$ if $\alpha < 1/2$.


\subsection{Zero-density and clustering}
For the Riemann zeta function, the Vinogradov--Korobov type estimate 
\begin{equation}
\label{Vinogradov--Korobov estimate}
	\zeta(s) \ll \abs{t}^{a(1-\sigma)^{3/2}}(\log\abs{t})^{2/3}
\end{equation}
yields\footnote{Vinogradov \cite{Vi58} claimed that $\zeta$ has no zeros for $\sigma > 1- c/(\log\abs{t})^{2/3}$ and seemed to imply this zero-free region follows from the estimate $\zeta(1+\I t) \ll (\log \abs{t})^{2/3}$ on the $1$-line alone. In view of our work, it is unlikely that Vinogradov had proofs for these claims. Indeed, for Beurling zeta functions, the bound \eqref{Vinogradov--Korobov estimate} cannot yield anything better than \eqref{Riemann zero-free} (up to the value of $c$), and moreover, bounds on the $1$-line alone do not produce any zero-free region at all, as $\zeta_{\mc}(1+\I t)$ is bounded for any choice of $f$ (see \eqref{bound zetac 1}-\eqref{bound zetac 2}).}
 a zero-free region of the form \eqref{Riemann zero-free}, as well as the zero-density estimate \eqref{Riemann zero-density}, which improves the Carlson-type zero density estimate $N(\sigma, T) \ll T^{A(1-\sigma)}(\log T)^{C}$ in the vicinity of the $1$-line. One can wonder whether a zero-free region $\sigma > 1 - f(\log t)$ with $f(u) \succ 1/u$ in itself already implies an improvement over the Carlson-type zero density estimate (apart from the trivial observation that $N\bigl(1-f(\log T)+\eps, T\bigr) = 0$). The Beurling zeta function from Section \ref{Proof of sharpness theorem} shows that this is not the case. Indeed, letting $b > 1$, it follows from the analysis in Subsection \ref{location zeros} that 
\[
	N\bigl(1-bf(\log T), T) = \Omega\Bigl( T^{(2b-1-o(1))f(\log T)}\Bigr),
\]
so that, apart from the values of the constants $A$ and $C$, no improvement of the Carlson-type estimate is possible.

\medskip

In \cite[Chapter 11]{Mo71}, Montgomery showed that the existence of a zero $\rho_{0} = \beta_{0}+\I\gamma_{0}$ of the Riemann zeta function with $\beta_{0}$ close to $1$ implies the existence of a cluster of zeros in the vicinity of $1+\I\gamma_{0}$ or $1+2\I\gamma_{0}$, with the size of this cluster bounded from below in terms of $(1-\beta_{0})^{-1}$. Later, a clustering result for Beurling zeta functions satisfying \eqref{theta-well-behaved} was worked out by him and Diamond and Vorhauer \cite[Theorem 2]{DMV}. Here, we focus on a recent stronger clustering result due to R\'ev\'esz \cite[Theorem 3]{Re22}. This result states a lower bound for a weighted count of the nearby zeros. For ease of reading, we present a corollary (which can be deduced from the theorem by integration by parts) dealing with the unweighted count directly (see also \cite[equation (11.20)]{Mo71}). For $T\ge0$ and $r>0$, we let
\[
	n(T,r) \coloneqq \#\bigl\{\rho: \abs{1+\I T-\rho} \le r \text{ and } \zeta_{\MP}(\rho)=0\bigr\} \quad \text{(counted with multiplicity)}.
\]
\begin{corollary}[R\'ev\'esz]
Let $\zeta_{\MP}(s)$ be the Beurling zeta function of a number system $(\MP, \MN)$ satisfying \eqref{theta-well-behaved}. Suppose $\zeta_{\MP}(\rho_{0})=0$, where $\rho_{0} = \beta_{0}+\I\gamma_{0}$ with $\gamma_{0}$ sufficiently large and $1-\beta_{0} < \frac{1-\theta}{40\log_{2}\gamma_{0}}$. There exist positive absolute constants $C_{1}$ and $C_{2}$ such that for every $\delta>0$ and every $\kappa\ge1$, either $1-\beta_{0} > C_{1}\delta$, or there exists an $r\in [\delta, 2(1-\theta)/3]$ for which
\begin{equation}
\label{clustering}
	n(\gamma_{0}, r) + n(2\gamma_{0}, r) \ge \frac{C_{2}}{\Gamma(\kappa+1)}\cdot \frac{r^{\kappa}}{\delta^{\kappa-1}(1-\beta_{0})}.
\end{equation}
\end{corollary}
It is worth pointing out that such clustering results can also be used to obtain zero-free regions from upper bounds on $\zeta_{\MP}(s)$ by Jensen's inequality, see the discussion in \cite[Chapter 11]{Mo71}. Nonetheless, Landau's method is more direct and (seemingly) more potent for this purpose.

Again, the examples in the style of Diamond, Montgomery, and Vorhauer show that this clustering result is sharp. Indeed, consider $\zeta_{1}(s)$ \eqref{zeta1} (or a discrete random approximation of it). It has a zero $\rho_{k} = \beta_{k}+\I\gamma_{k}$ with $1-\beta_{k} = 1/\ell_{k}$. On the other hand, using \eqref{zeros G} one readily sees that
\[
	n(\gamma_{k}, r) \ll \ell_{k} r, \quad n(2\gamma_{k}, r) = 0.
\]
Comparing with \eqref{clustering}, we infer that for $\delta \ge 1/(C_{1}\ell_{k})$ and $r=r(\delta, \kappa)$, 
\[
	\frac{C_{2}}{\Gamma(\kappa+1)}\cdot \frac{r^{\kappa}\ell_{k}}{\delta^{\kappa-1}} \le  n(\gamma_{k}, r) + n(2\gamma_{k}, r) \ll \ell_{k} r,
\]
showing that $C_{2}$ cannot be chosen arbitrarily large in \eqref{clustering}. Furthermore, in this example one necessarily has $r \ll_{\kappa} \delta$.

\section*{Acknowledgements} The idea for this article arose during the Focused Workshop on Beurling Number Systems, which took place at the end of July 2024 at the Alfr\'ed R\'enyi Institute of Mathematics, Budapest, Hungary, with the support of the Erd\H{o}s Center. The author would like to thank the other participants Gregory Debruyne, Titus Hilberdink, J\'anos Pintz, Szil\'ard R\'ev\'esz, Imre Ruzsa, and Jasson Vindas for the many interesting discussions.

\end{document}